\documentclass[11pt]{article}

\usepackage{amsmath,amssymb,amscd,amsxtra,mathrsfs}
\usepackage{amsthm}
\usepackage{mathtools}
\usepackage{float}
\usepackage{amsfonts}
\usepackage{hyperref}	
\usepackage{xcolor}	
\usepackage{bm}    
\usepackage{graphicx}
\usepackage{multirow}
\usepackage{caption}
\usepackage{subcaption}
\usepackage{xcolor}

\newcommand{\R}{\mathbb{R}}
\newcommand{\E}{\mathcal{E}}

\newcommand{\T}{\mathcal{T}}
\newcommand{\F}{\mathcal{F}}

\newcommand{\K}{\mathcal{K}_{\text{vor}}}
\newcommand{\M}{\mathscr{M}}
\newcommand{\A}{\textbf{a}}

\newcommand{\bsi}{{\bf a}}
\newcommand{\mL}{\mathcal{K}_{\textrm{vor}}}
\newcommand{\ck}{{\kappa_0}}
\newcommand{\si}{a}
\newcommand{\mV}{\bm{V}}
\newcommand{\hold}{A}
\newcommand{\vp}{\varphi}
\newcommand{\Khold}{\mathcal{K}_{\hold}}
\newcommand{\C}{\mathcal{C}}
\newcommand{\trp}{Y}
\newcommand{\dop}{X}
\newcommand{\dbsi}{\delta\bsi}
\newcommand{\bta}{\theta}
\newcommand{\Om}{V}
\newcommand{\tz}{t_0}
\newcommand{\lsf}{\phi}
\newcommand{\mm}{\mathscr{M}}
\newcommand{\dsi}{\delta\si}

\newtheorem{theorem}{Theorem}[section]

\newtheorem{definition}{Definition}[section]
\newtheorem{assump}{Assumption}

\begin{document}

\title{Optimization of centroidal Voronoi tessellations\thanks{This work has been partially supported by the Brazilian agencies FAPESP (grants 2013/07375-0, 2022/05803-3, and 2023/08706-1) and CNPq (grant 302073/2022-1).}}

\author{
    Ernesto G. Birgin\thanks{Department of Computer Science, Institute of
    Mathematics and Statistics, University of S\~ao Paulo, Rua do
    Mat\~ao, 1010, Cidade Universit\'aria, 05508-090, S\~ao Paulo, SP,
    Brazil. e-mail: egbirgin@ime.usp.br}
\and 
    Juan S. C. Franco\thanks{Department of Applied Mathematics, Institute of
    Mathematics and Statistics, University of S\~ao Paulo, Rua do
    Mat\~ao, 1010, Cidade Universit\'aria, 05508-090, S\~ao Paulo, SP,
    Brazil. e-mail: jcastanof@ime.usp.br}
\and 
    Antoine Laurain\thanks{Faculty of Mathematics, University of Duisburg-Essen, Thea-Leymann-Str. 9, 45127, Essen, Germany, e-mail: antoine.laurain@uni-due.de}
}

\date{February 3, 2025}

\maketitle

\begin{abstract}
In this paper, we investigate the optimization of Centroidal Voronoi Tessellations (CVT) under geometric constraints. For this purpose, we minimize a linear combination of the standard CVT energy functional with terms involving geometric attributes such as area and perimeter. The derivative of the objective functional with respect to the position of the generators is computed using techniques of shape calculus and sensitivity analysis of minimization diagrams. Several numerical experiments are presented to explore the geometric constraints of cells with identical areas, cells without small edges, and density-based distributions of cells.\\


\end{abstract}

\section{Introduction}

Centroidal Voronoi Tessellations (CVTs) are a special case of Voronoi tessellations in which each site coincides with the centroid of its Voronoi cell. 
They play an important role in various applications in science and engineering, including image processing, data compression, and numerical approximations of partial differential equations, particularly for mesh generation.
We refer to~\cite{du} for a comprehensive survey of applications. 
 CVT has been generalized to very broad settings, such as CVT of surfaces or line segments, distance metrics, and discrete point sets \cite{NMTMA-3-119}.

A standard approach to computing CVTs is the deterministic Lloyd's algorithm \cite{du,MR651807}, which is a fixed-point algorithm. 
The convergence of Lloyd's algorithm is slow, but several improvements have been introduced, such as the Lloyd-Newton method \cite{MR2209432} and variants that employ quasi-Newton methods, such as the limited-memory BFGS (LBFGS) method \cite{MR3315273,liu:inria-00547936}. CVTs can also be constructed using probabilistic methods such as MacQueen's algorithm \cite{macqueen}.

As shown in \cite{liu:inria-00547936}, a CVT is a critical point of the CVT energy function and can be constructed using derivative-based optimization methods since the energy function is $C^2$ within a convex domain. The derivative of the energy and the Lloyd map were computed in \cite{du} and used in subsequent papers such as \cite{bogosel2024numerical, liu:inria-00547936}. The use of these derivatives for gradient-based, Newton and quasi-Newton methods is also discussed in \cite{du}.

It was shown in \cite{birgin:01} that this type of sensitivity analysis for Voronoi diagrams can be recast in the much more general framework of minimization diagrams. Minimization diagrams are a broad class of diagrams whose cells are defined via the lower envelope of a set of graphs of functions. Many relevant diagrams, such as Voronoi diagrams or power diagrams, are special cases of minimization diagrams.  These theoretical results allow us to consider the optimization of any differentiable criterion depending on the geometric properties of the minimization diagram. This includes CVT as a special case, but also various other criteria such as perimeter, area, angles, and other geometric quantities. A useful feature of this general theory is the handling of fixed elements of the geometry, such as the boundary of the domain, which are usually not considered in the sensitivity analysis of Voronoi diagrams. In \cite{birgin:01}, the authors presented several numerical applications for the optimization of Voronoi diagrams with the aim of obtaining specific properties such as equal cell areas or edges of equal length. These properties are desirable for improving mesh quality using quality measures, see \cite{du,HRK}.

In this paper, we extend these results to the optimization of CVTs. We consider objective functions that are linear combinations of the CVT energy with additional terms that enforce geometric constraints. The goal is to demonstrate how one can control the geometric properties of CVTs. This can also be seen as minimizing the CVT energy with additional constraints in the form of penalizations. We focus on two types of geometric constraints. The first constraint is of the area type, forcing cells to have identical areas. The second type of constraint is a perimeter constraint, where the goal is to avoid cells with small edges.

In Section~\ref{sec:grad} we introduce the notation and tools necessary for the sensitivity analysis of Voronoi diagrams, and then apply these results to the computation of the gradient of the CVT and other relevant shape functions. Two different ways to compute the gradients are presented. Numerical experiments are given in Section~\ref{sec:num}. In the experiments, we show how to compute CVTs where all cells have the same area, cells without edges considered small, and cells whose area is governed by a given function. Conclusions and lines of future work are presented in the last section.

\section{The CVT energy function and its gradient}\label{sec:grad}

In this section, we first recall the main results of \cite{birgin:01} that are necessary for the sensitivity analysis of Voronoi diagrams. We then introduce the CVT energy function and use these results to compute its gradient. Consider the set $\hold := \{ x\in\R^2 :\ \vp(x)<0\}$ with $\vp(x) := \min_{\ell\in\Khold} \vp_\ell(x)$ and $\vp_\ell\in\C^\infty(\R^2,\R)$ for all $\ell\in \Khold := \{\ck+1, \dots,\ck+\kappa_1\}$ for some given $\ck,\kappa_1$. For $\ell\in\Khold$, introduce the set $\partial_\ell \hold: = \{x\in\partial \hold :\ \vp_\ell(x)=0\}$; then we have $\partial \hold = \cup_{\ell\in\Khold} \partial_\ell \hold$. Denote by $\T_{\partial \hold}$ the finite set of corners of $\hold$. Let $\mL = \{1,\dots,\ck\}$ be a set of indices, $\bsi = \{\si_k\}_{k\in\mL}$ be a set of points in the plane, the so-called {\it sites}, and let $\mV(\bsi) := \{V_i(\bsi)\}_{i\in\mL}$ be the Voronoi diagram associated with $\bsi$, where the cells of the diagram are defined by
\[
V_i(\bsi):= \left\{x \in \hold \text{ such that } \|x - \si_i\| \leq  \|x - \si_j\| \text{ for all } j \in\mL\setminus\{i\} \right\}.
\]
We will also need the following notation. The Euclidean norm is denoted by $\|\cdot\|$. For $x,y\in \R^n$, $x\cdot y = x^\top y\in \R$; $x\otimes y=xy^{\top}\in \R^{n\times n}$. We use $y^{\perp}:= Ry$, for $y\in \R^2$, where $R$ is a rotation matrix of angle $\pi/2$ with respect to a counterclockwise orientation. The transpose of a matrix $M$ is indicated by $M^\top$. The gradient with respect to $x \in \R^2$ of a function $\psi : \R^2 \to \R$ is denoted $\nabla_x \psi$ and is a column vector. The Jacobian matrix is denoted $D_x\psi$. The gradient with respect to \textnormal{\A} of a function $G:\R^{q\kappa_0} \rightarrow \R$ is $\nabla G$. The Jacobian matrix with respect to $\textnormal{\A}$ of a function $G:\R^{q\kappa_0} \rightarrow \R^n $ is denoted by $DG$. For a sufficiently smooth set $S \subset \R^2$, $\overline{S}$ denotes its closure, $|S|$ its perimeter if $S$ is one-dimensional or its area if $S$ is two-dimensional. 

\subsection{Sensitivity analysis of Voronoi diagrams}\label{sec:2.1}

Let $\mV(\bsi+t\delta\A) := \{V_k(\bsi+t\delta\A)\}_{k\in\mL}$ be a perturbed Voronoi diagram. In what follows, we describe the tools and notation needed for the sensitivity analysis of $\mV(\bsi+t\delta\A)$ with respect to~$t$.

\begin{definition}[interior vertices and edges]
For $\{ i,j,k \} \subset \K$, we define the set of \textit{ inner vertices} $Y_{ijk}(t)\coloneq \overline{V_i(\bsi+t\delta\bsi)}\cap\overline{V_j(\bsi+t\delta\bsi)}\cap \overline{V_k(\bsi+t\delta\bsi)}$, that is, points in $A$ at the intersection of three cells, and write $Y_{ijk}\coloneq Y_{ijk}(0)$.  
For $k\in \mL\setminus\{i\}$, $E_{ik}(\bsi+t\dbsi):= \overline{\Om_i(\bsi+t\dbsi)}\cap\overline{\Om_k(\bsi+t\dbsi)}$ denotes an {\it interior edge} of $\mV(\bsi+t\dbsi)$.
\end{definition}

\begin{definition}[boundary vertices and edges]
For $\{ i,j \} \subset \K$ and $\ell\in \mathcal{K}_A$, we define the set of \textit{boundary vertices}  $X_{ij\ell}(t)\coloneq \overline{V_i(\bsi+t\delta\bsi)}\cap\overline{V_j(\bsi+t\delta\bsi)}\cap \partial_{\ell}A$, i.e., points on $\partial A$ at the intersection of two cells, and write $X_{ij\ell}\coloneq X_{ij\ell}(0)$.
For $k\in \mL\setminus\{i\}$ and $\ell\in\Khold$, $E_{ik}(\bsi+t\dbsi):= \overline{\Om_i(\bsi+t\dbsi)}\cap\partial_\ell \hold$ denotes a {\it boundary edge} of $\mV(\bsi+t\dbsi)$.
\end{definition}

We now recall the results of \cite{birgin:01}, which will allow us to perform the sensitivity analysis of Voronoi diagrams. In \cite{birgin:01}, a theoretical framework for the analysis of the sensitivity of minimization diagrams was developed, and the particular case of Voronoi diagrams was discussed in \cite[\S4]{birgin:01}. 
A key aspect of this theory is establishing a set of geometric assumptions that avoid degenerate cases and under which the sensitivity analysis can be performed. 
Specifically, these assumptions ensure that the interior vertices $\trp_{ijk}$ of the Voronoi diagram belong to at most three cells. Additionally, they eliminate trivial cases where two cells with different indices are identical.
These assumptions were formulated in the more general context of minimization diagrams, so here we provide a simpler formulation in the case of Voronoi diagrams, which summarizes the discussion in \cite[\S4]{birgin:01}.

\begin{assump}\label{assump1:MDbd}
Suppose that:
\begin{itemize}
\item (Non-degeneracy of interfaces)
There holds $\|\nabla_x \vp_\ell(x)\|>0$ for all $x\in \partial_\ell \hold$ and for all $\ell\in\Khold$,  and $\|\si_i -\si_j\|>0$ for all $\{i,j\}\subset\mL$.
\item (Non-degeneracy of vertices) For all  $\{i,j,k\}\subset \mL$ such that $\trp_{ijk}\neq\emptyset$ we have
$(\si_j-\si_i)^\perp\cdot (\si_k-\si_i) \neq 0$ and $\trp_{ijk}\cap \overline{V_m(\bsi)}=\emptyset$, for all $m\in\mL\setminus\{i,j,k\}$.
In addition, for all $\{i,j\}\subset \mL$ and $\ell\in\Khold$ and all $v\in  \dop_{ij\ell}$ we have
$(\si_j-\si_i)^\perp\cdot \nabla\vp_\ell(v) \neq 0$, $v\cap  \overline{V_m(\bsi)}=\emptyset$, for all $m\in\mL\setminus\{i,j\}$, and $v\cap \mathcal{T}_{\partial A}=\emptyset$, where $\mathcal{T}_{\partial A}$ is the finite set of corners of $A$.
\end{itemize}
\end{assump}

Under these geometric assumptions, the set $\trp_{ijk}$ contains at most one point, but $\dop_{ij\ell}$ may contain multiple points. The key to the sensitivity analysis of the perturbed Voronoi diagram $\mV(\bsi+t\delta\A)$ is the computation of the derivatives of the vertices $Y_{ijk}(t)$ and $X_{ijk}(t)$. This is essentially an application of the implicit function theorem under Assumption~\ref{assump1:MDbd}.

\begin{theorem} \label{theo:1}
Suppose Assumption~\ref{assump1:MDbd} holds and $|Y_{ijk}|=1$ for some $\{ i,j,k \} \subset \K$. Then, denoting $v=Y_{ijk}$, there exists $\tau_1>0$ and a unique smooth function $z_v:[0,\tau_1]\to \R^2$ satisfying $z_v(0)=v$, $z_v(t)=Y_{ijk}(t)$ for all $t\in [0,\tau_1]$ and
\begin{equation}\label{eq:29}
z'_v(0)= M_v(j,k,i)\delta a_i + M_v(k,i,j)\delta a_j + M_v(i,j,k)\delta a_k,
\end{equation}
where
\[
M_v(i,j,k) \coloneq  \frac{(a_i-a_j)^\perp \otimes (v-a_k)^\top}{Q(i,j,k)}
\]
and 
\begin{equation*}
Q(i,j,k) \coloneq \det 
\begin{pmatrix}
(a_j-a_i)^{\top} \\
(a_k-a_i)^{\top}
\end{pmatrix}.
\end{equation*}
\end{theorem}

\begin{proof}
See \cite[Theorem 7]{birgin:01}.
\end{proof}

\begin{theorem} \label{theo:2}
Suppose Assumption~\ref{assump1:MDbd} holds and let $\{ i,j \} \subset \K, \, \ell\in \mathcal{K}_A$. Then $X_{ij\ell}$ is finite, $X_{ij\ell}\in \partial A \setminus \mathcal{T}_{\partial A}$
, and there exists $\tau_1>0$ such that for all $v\in X_{ij\ell}$ there exists a unique smooth function $z_v:[0,\tau_1]\to \R^2$ satisfying $z(0)=v$, $\varphi_{\ell}(z_v(t))=0$ for all $t\in[0,\tau_1]$, and 
\[
X_{ij\ell}(t) = \bigcup_{v\in X_{ij\ell}}\{ z_v(t) \} \text{ for all } t\in [0,\tau_1].
\]
In addition we have
\begin{equation}\label{eq:32}
z'_v(0)=\M^{\ell}_v(j,i)\delta a_i + \M^{\ell}_v(i,j)\delta a_j
\end{equation}
with
\[
\M^{\ell}_v(j,i) \coloneq  \frac{-\nabla_x\varphi_{\ell}(v)^{\perp}\otimes (v-a_i)^{\top}}{\det\begin{pmatrix}
(a_j-a_i)^{\top} \\
\nabla_x \varphi_{\ell}(v)^{\top}\end{pmatrix}}.
\]
\end{theorem}

\begin{proof}
See \cite[Theorem 8]{birgin:01}.
\end{proof}

Under Assumption~\ref{assump1:MDbd}, the motion of each Voronoi cell $\Om_i(\bsi+t\dbsi)$, $i\in\mL$, can be parameterized by a bi-Lipschitz mapping $T(\cdot, t)$, such that its derivative $\bta := \partial_t T(\cdot,0)$ can be described explicitly as a function of the sites $\bsi$. The explicit expression of $\bta$ at the vertices is a consequence of Theorems~\ref{theo:1} and ~\ref{theo:2}. This parameterization is described in the following theorem, which is a particular case of  \cite[Theorem~5]{birgin:01}.

\begin{theorem} \label{thm01:walls}
Let $i\in\mL$ and suppose Assumption~\ref{assump1:MDbd} holds. Then there exist $\tz>0$ and a mapping $T:\overline{\Om_i(\bsi)} \times [0,\tz] \to \R^2$ satisfying $T(\Om_i(\bsi),t) = \Om_i(\bsi+t\dbsi)$, $T(E_{ik}(\bsi),t)=E_{ik}(\bsi+t\dbsi)$ for all $k\in\mL\setminus\{i\}$, $T(E_{i\ell}(\bsi),t)=E_{i\ell}(\bsi+t\dbsi)$ for all $\ell\in\Khold$,  $T(\partial \Om_i(\bsi),t) = \partial \Om_i(\bsi+t\dbsi)$ and $T(\cdot, t): \overline{\Om_i(\bsi)} \to \overline{\Om_i(\bsi+t\dbsi)}$ is  bi-Lipschitz   for all $t\in [0,\tz]$. In addition we have
\begin{equation} \label{der_normal:walls} 
\begin{array}{rcl}
\bta(x)\cdot\nu(x) &=& \frac{ \nabla_{\si} \lsf(x,\si_k)\cdot\delta\si_k - \nabla_{\si} \lsf(x,\si_i)\cdot\delta\si_i}{\|\nabla_x\lsf(x,\si_k) - \nabla_x\lsf(x,\si_i)\|}\text{ for all } x\in E_{ik}(\bsi) \text{ and all }  k\in\mL\setminus\{i\}, \\ [2mm]
\bta(x)\cdot\nu(x) &=& 0\text{ for all } x\in E_{i\ell}(\bsi)  \text{ and for all } \ell\in\Khold,\\[2mm]
\theta(v)\cdot\tau(v) &=& (M_v(j,k,i) \dsi_i + M_v(k,i,j) \dsi_j + M_v(i,j,k) \dsi_k)\cdot \tau(v)\\
&& \text{ for all } v\in\trp_{ijk},  \{i,j,k\}\subset \mL,\\[2mm]
\theta(v)\cdot\tau(v) &=& ( \mm_v^\ell(j,i) \dsi_i + \mm_v^\ell(i,j) \dsi_j)\cdot \tau(v)\text{ for all } v\in\dop_{ij\ell}, \{i,j\}\subset \mL, \ell\in\Khold.
\end{array}
\end{equation}
where $\bta := \partial_t T(\cdot,0)$, $\nu$ is the outward unit normal vector to $\Om_i(\bsi)$, $\lsf(x,\si):= \|x-\si\|^2$, and $\tau$ is a tangent vector to $\partial V_i(\bsi)$. 
\end{theorem}

Theorems~\ref{theo:1}, \ref{theo:2}, and~\ref{thm01:walls} permit the computation of various relevant geometric quantities related to Voronoi diagrams. The derivative of area and edge integrals is of particular interest and has been performed in \cite[Sections~3 and 4]{birgin:01}; here we give a brief summary of these results. 
Starting with the case of an area integral, consider 
\begin{equation}\label{perturbedG1}
G_1(\A+t\dbsi)\coloneq \int_{V_i(\A+t\dbsi)}f(x)\,dx,
\end{equation}
where $f\in C^1(\overline{A},\R)$. Then, using Theorem~\ref{thm01:walls}, we have $V_i(\A+t\dbsi) = T(V_i(\A),t)$ and we can apply a change of variable to transform the integral in \eqref{perturbedG1} to an integral on $V_i(\A)$.
Next we can differentiate the obtained expression with respect to $t$, use the divergence theorem to transform the integral on $V_i(\A)$ to an integral on $\partial V_i(\A)$, and use the first equality in \eqref{der_normal:walls} to obtain
\begin{align}\label{eq:cellint}
\begin{split}
\nabla G_1(\A)\cdot \delta \A & = \sum_{E\in \E^{\text{int}}_i}\frac{\delta a_i}{\|a_i - a_{k(i,E)}\|}\cdot \int_E f(x)(x-a_i)\,dx\\
& - \sum_{E\in \E^{\text{int}}_i} \frac{\delta a_{k(i,E)}}{\|a_i - a_{k(i,E)}\|}\cdot \int_E f(x)(x-a_{k(i,E)})\,dx. 
\end{split}
\end{align}

Next, consider the edge integral function defined by
\begin{equation}\label{eq:perturbedG2}
G_2(\A+t\dbsi)\coloneq \int_{E(\A+t\dbsi)}f(x)\,dx,
\end{equation}
where $f\in C^1(\overline{A},\R)$. Here, the edge $E(\A)$ can either be an interior edge given by $E(\A)=\overline{V_i(\A)}\cap\overline{V_k(\A)},\,\{ i,k \} \subset \K$, or a boundary edge given by $E(\A)=\overline{V_i(\A)}\cap \partial_{\ell}A, \, \ell \in \mathcal{K}_A$. In a similar way as for the area integral, using Theorem~\ref{thm01:walls} we obtain the following expression for the gradient:
\begin{equation}\label{eq:35}
\nabla G_2(\A) \cdot \delta \A = \F(i,w_E)\cdot \tau (w_E) - \F(i,v_E)\cdot \tau (v_E),
\end{equation}
where $\tau$ is here the tangent vector to $\partial V_i(\textnormal{\A})$ with respect to a counterclockwise orientation and
\begin{equation} \label{eq:36}
\F(i,v) = 
\begin{cases}
M_v(j,k,i)\delta a_i + M_v(k,i,j)\delta a_j + M_v(i,j,k)\delta a_k, & \text{if } v\in Y_{ijk}, \\
\M^{\ell}_v(j,i)\delta a_i + \M^{\ell}_v(i,j)\delta a_j, & \text{if } v \in X_{ij\ell},\\
0, & \text{if } v\in \T_{\partial A}.
\end{cases} 
\end{equation}
Note that in \eqref{eq:36}, the indices $j,k$ in $Y_{ijk}$ and the index $j$ in $X_{ij\ell}$ actually depend on the index $i$ and on the vertex $v$. These indices may be uniquely determined by choosing a counterclockwise orientation of the cells around the vertex $v$.

\subsection{CVT energy function}

Let $\rho: A\to \R$, $\rho >0$, be a given density function. We work with the CVT energy function \cite{bogosel2024numerical,du,liu:inria-00547936} defined by 
\begin{equation}\label{eq:mfun2}
G(\A)\coloneq \frac{1}{\kappa_0}\sum^{\kappa_0}_{i=1} G_i(\A) \text{ with } G_i(\A) = \displaystyle\int_{V_i(\A)}\rho(x) \|x-a_i\|^2\,dx,
\end{equation}
recalling that $\bsi = \{a_i\}_{i=1}^{\ck}$ are the sites generating the cells $V_i(\A)$. In \cite{du,liu:inria-00547936} it is shown that the critical points of the function $G(\A)$ generates a CVT. The calculation of the gradient of $G(\A)$ is standard and has been performed, for example, in \cite{du,liu:inria-00547936}. Here, in Sections~\ref{subsec:2.3} and~\ref{subsec:2.4} we present the calculation of the gradient of $G(\A)$ using our framework, which will also allow us to supplement $G(\A)$ with other geometric terms in the numerical experiments. We present two ways to calculate the gradient of $G(\A)$. 
First, we compute $\nabla G(\A)$ directly  using \eqref{eq:cellint}, and second, we provide first an explicit formula for $G_i(\A)$ and $G(\A)$ as a function of the vertices and $a_i$, and compute the gradient of that formula.
We then discuss the differences between these two formulas for applications.

For numerical purposes, \eqref{eq:mfun2} needs to be evaluated. In the general case, quadrature rules need to be employed. 
In some particular cases, an explicit expression of $G(\A)$ can be obtained. Let $T(a_i,v,w)$ be the triangle with vertices $a_i,v,w$, where $w$ is the neighbor vertex of $v$ in counterclockwise direction, and define $\mathcal{V}_i$ as the set of vertices of $V_i(\A)$. Then we have the partition $V_i(\A)=\displaystyle\bigcup_{v\in \mathcal{V}_i}T(a_i,v,w)$. Thus
\begin{equation*}
G_i(\A) = \int_{V_i(\A)}\rho(x)\|x-a_i\|^2\,dx = \sum_{v\in \mathcal{V}_i}\int_{T(a_i,v,w)}\rho(x)\|x-a_i\|^2\,dx. 
\end{equation*}
Now, we consider the transformation $\phi: \R^2\to\R^2$ defined as follows,
\begin{equation*}
\phi(\xi,\lambda) = a_i + \xi(v-a_i) + \lambda(w-a_i) \text{ with } \xi,\lambda\geq 0 \text{ and } \xi+\lambda=1.
\end{equation*}
Using a change of variables, we have
\begin{equation*}
\int_{T(a_i,v,w)}f(x)\,dx = \int^1_{\xi=0}\int^{1-\xi}_{\lambda=0}f(\phi(\xi,\lambda))|J(v,w,\si_i) |\,d\lambda\,d\xi,
\end{equation*}
with the Jacobian matrix $J(v,w,\si_i) =D\phi=(v-a_i | w-a_i)\in \R^{2\times2}$ and
\begin{equation} \label{eq:detphi}
\begin{array}{rcl} 
|J(v,w,\si_i)| 
&=& 
\begin{vmatrix}
v_{1} - a_{i,1} & w_{1} - a_{i,1} \\
v_{2} - a_{i,2} & w_{2} - a_{i,2} 
\end{vmatrix} \\ [4mm]
&=& (v_{1} - a_{i,1})(w_{2} - a_{i,2}) - (v_{2} - a_{i,2})(w_{1} - a_{i,1}) \\ [2mm]
&=& -(v-a_i)\cdot(w-a_i)^\perp. 
\end{array}
\end{equation}
For $f(x) = \rho(x)\|x-a_i\|^2$, we get 
\begin{align*}
& \int^1_{0}\int^{1-\xi}_{0}\rho(\phi(\xi,\lambda))\|\phi(\xi,\lambda)-a_i\|^2|J(v,w,\si_i) |\,d\lambda d\xi\\
& =  |J(v,w,\si_i) |\int^1_{0}\int^{1-\xi}_{0}\rho(\phi(\xi,\lambda))\|\xi(v-a_i) + \lambda(w-a_i)\|^2\,d\lambda d\xi \\
& = |J(v,w,\si_i) |\int^1_{0}\int^{1-\xi}_{0}\rho(\phi(\xi,\lambda))(\xi^2\|v-a_i\|^2+\lambda^2\|w-a_i\|^2+2\lambda\xi(v-a_i)\cdot(w-a_i)) \,d\lambda d\xi.
\end{align*}
To simplify this expression further, we need to choose a specific class of functions $\rho$ such as polynomials. Let us consider the particular case $\rho \equiv 1$, which results in
\begin{equation*}
\int^1_{0}\int^{1-\xi}_{0}\|\phi(\xi,\lambda)-a_i\|^2|J(v,w,\si_i) |\,d\lambda d\xi = \frac{|J(v,w,\si_i) |}{12}(\|v-a_i\|^2+(v-a_i)\cdot(w-a_i)+\|w-a_i\|^2).
\end{equation*}
Finally,
\begin{equation*}
G_i(\A) = \sum_{v\in\mathcal{V}_i} \frac{|J(v,w,\si_i) |}{12}(\|v-a_i\|^2+(v-a_i)\cdot(w-a_i)+\|w-a_i\|^2),
\end{equation*}
and this yields the following explicit formula of $G(\A)$ in the case $\rho\equiv 1$:
\begin{equation}
G(\A) = \frac{1}{\kappa_0}\sum^{\kappa_0}_{i=1}\left(\sum_{v\in\mathcal{V}_i} \frac{|J(v,w,\si_i) |}{12}(\|v-a_i\|^2+(v-a_i)\cdot(w-a_i)+\|w-a_i\|^2)\right).
\label{eq:explicitG}
\end{equation}

\subsection{Computing the gradient of the integral form of \texorpdfstring{$G(\A)$}{G(a)}}\label{subsec:2.3}

Here, we compute the value of the gradient of the integral form \eqref{eq:mfun2} of $G(\A)$ using \eqref{eq:cellint}. 
First, we have
\begin{equation*}
\nabla G(\A) = \nabla \left(\frac{1}{\kappa_0} \sum^{\kappa_0}_{i=1}G_i(\A)  \right) = \frac{1}{\kappa_0}\sum^{\kappa_0}_{i=1} \nabla G_i(\A).
\end{equation*}
Let $\E^{\text{int}}_i$ be the set of the interior edges of the cell $V_i(\A)$ and $k(i,E)$ be the index such that $E = \overline{V_i(\A)} \cap \overline{V_{k(i,E)}(\A)}$. Applying \eqref{eq:cellint}, we get 
\begin{align*}
\nabla G_i(\A)\cdot \delta \A &= \nabla\left( \int_{V_i(\A)}\rho(x)\|x-a_i\|^2\,dx \right) \cdot \delta \A \\
& = \sum_{E\in \E^{\text{int}}_i}\frac{\delta a_i}{\|a_i - a_{k(i,E)}\|} \cdot \underbrace{\int_E \rho(x) \|x-a_i\|^2(x-a_i)\,dx}_{=: I^1_E} \\ 
& \quad - \sum_{E\in \E^{\text{int}}_i}\frac{\delta a_{k(i,E)}}{\|a_i - a_{k(i,E)}\|} \cdot \underbrace{\int_E \rho(x)\|x-a_i\|^2(x-a_{k(i,E)})\,dx}_{=:  I^2_E}\\
& \quad -2 \delta a_i \cdot \int_{V_i(\A)}\rho(x) (x-a_i) \,dx.
\end{align*}
Note that the last term is not present in \eqref{eq:cellint} but appears as the integrand $\rho(x)\|x-a_i\|^2$ depends on $a_i$.

Introducing $c_i\coloneq\displaystyle \frac{\int_{V_i(\A)}\rho(x)x\,dx}{\int_{V_i(\A)}\rho(x)\,dx}$ the centroid of the cell $V_i(\A)$, we obtain
\begin{align*}
\nabla G_i(\A)\cdot \delta \A 
& = 2 \delta a_i \cdot (a_i - c_i)\int_{V_i(\A)}\rho(x)\,dx + \sum_{E\in \E^{\text{int}}_i} \frac{\delta a_i\cdot I^1_E - \delta a_{k(i,E)} \cdot I^2_E}{\|a_i - a_{k(i,E)}\|}.
\end{align*}
This yields
\[
\nabla G(\A)\cdot\delta \A = \frac{1}{\kappa_0}\sum^{\kappa_0}_{i=1}\sum_{E\in \E^{\text{int}}_i}\frac{\delta a_i\cdot I^1_E - \delta a_{k(i,E)} \cdot I^2_E}{\|a_i - a_{k(i,E)}\|} +2 \delta a_i \cdot (a_i - c_i)\int_{V_i(\A)}\rho(x)\,dx.
\]
Now we rearrange the term
\begin{align*}
\sum^{\kappa_0}_{i=1}\sum_{E\in \E^{\text{int}}_i}\frac{\delta a_{k(i,E)}\cdot I^2_E}{\|a_i - a_{k(i,E)}\|} 
\end{align*}
by summing over all fixed indices $k$ such that $k = k(i,E)$ for some index $i$ and edge $E$.
Then $a_i$ becomes $a_{i(k,E)}$ which yields
\begin{align*}
\sum^{\kappa_0}_{i=1}\sum_{E\in \E^{\text{int}}_i}\frac{\delta a_{k(i,E)}\cdot I^2_E}{\|a_i - a_{k(i,E)}\|} 
& = \sum^{\kappa_0}_{k=1}\sum_{E\in \E^{\text{int}}_k}\frac{\delta a_k\cdot \displaystyle\int_E\rho(x)\|x - a_{i(k,E)}\|^2(x-a_k)\,dx }{\|a_{i(k,E)} - a_{k}\|}.
\end{align*}
Using the property $\|x-a_{i(k,E)}\|=\|x-a_{k}\|$ on $E$ and changing the notation for indices, we end up with 
\begin{align*}
\sum^{\kappa_0}_{i=1}\sum_{E\in \E^{\text{int}}_i}\frac{\delta a_{k(i,E)}\cdot I^2_E}{\|a_i - a_{k(i,E)}\|}  
& =  \sum^{\kappa_0}_{i=1}\sum_{E\in \E^{\text{int}}_i}\frac{\delta a_i\cdot I^1_E}{\|a_i - a_{k(i,E)}\|}.
\end{align*}

Thus, the terms depending on $I^1_E$ and $I^2_E$ cancel out in $\nabla G(\A)\cdot \delta \A$ and
\begin{equation}\label{grad1}
\nabla G(\A)\cdot \delta \A = \frac{1}{\kappa_0}\sum^{\kappa_0}_{i=1} 2 \delta a_i \cdot (a_i - c_i)\int_{V_i(\A)}\rho(x)\,dx, 
\end{equation} 
which is the standard formula in the literature, see for instance \cite{bogosel2024numerical, du, 10.1007/BFb0008901}. 
The fact that $I^1_E$ and $I^2_E$ cancel out is also written in \cite[p.278]{10.1007/BFb0008901} for instance.

\subsection{Explicit gradient computation in the constant density case} \label{subsec:2.4}

To obtain the expression~\eqref{grad1} of $\nabla G(\A)$, we have used \eqref{eq:cellint}, which is based on Theorem~\ref{thm01:walls}, using the shape calculus techniques of \cite{birgin:01}. In the special case where $\rho$ is constant, $\nabla G(\A)$ can be obtained in a simpler way, without using Theorem~\ref{thm01:walls}, by directly differentiating \eqref{eq:explicitG}. The purpose of this section is to perform this calculation and discuss the formula obtained. Note that one could also perform an explicit calculation of $\nabla G(\A)$ for specific  classes of functions $\rho$, such as polynomials.

Let $\E_i$ be the set of edges of $V_i(\A)$. In \eqref{eq:explicitG} the sum is over $\mathcal{V}_i$, the set of vertices of $V_i(\A)$. We transform it into a sum over $E\in \E_i$ in order to use the results of Section~\ref{sec:2.1}. We also write $v_E$, $w_E$ instead of $v$,$w$, recalling that $w$ is the neighbor vertex of $v$ in counterclockwise orientation. Differentiating \eqref{eq:explicitG}, we thus obtain 
\begin{equation}
\nabla G(\A)\cdot \delta \A=\frac{1}{\kappa_0}\sum^{\kappa_0}_{i=1}\left(\frac{1}{12}\sum_{E\in \E_i}\sigma\nabla \gamma + \gamma\nabla \sigma\right)\cdot \delta \A,
\label{eq:gradexpl}
\end{equation}
where $\gamma \coloneq |J(v_E,w_E,\si_i)|$ and $\sigma\coloneq\|v_E-a_i\|^2+(v_E-a_i)\cdot(w_E-a_i)+\|w_E-a_i\|^2$. First, we compute 
\begin{align*}
\nabla \gamma \cdot \delta \A = \nabla (|J(v_E,w_E,\si_i)|)\cdot\delta \A = \text{sign}(J(v_E,w_E,\si_i))\nabla (J(v_E,w_E,\si_i)) \cdot \delta \A.
\end{align*}
Using \eqref{eq:detphi}, we have 
\begin{align*}
\nabla(J(v_E,w_E,\si_i)) \cdot \delta \A 
& = -\nabla ((v_E-a_i)\cdot(w_E-a_i)^\perp)\cdot\delta\A \\
& = - ((D(v_E-a_i))^\top(w_E-a_i)^\perp + (D(w_E-a_i)^\perp)^\top(v_E-a_i))\cdot \delta \A \\
& = -D(v_E-a_i)\delta \A \cdot (w_E-a_i)^\perp - D(w_E-a_i)^\perp \delta \A \cdot (v_E-a_i) \\
& = -D(v_E-a_i)\delta \A \cdot (w_E-a_i)^\perp + D(w_E-a_i)\delta \A \cdot (v_E-a_i)^\perp \\
& = -Dv_{E}\delta \A \cdot w^\perp_E + Dv_{E}\delta \A \cdot a^\perp_i + Da_i\delta\A \cdot w^\perp_E\\
& \quad + Dw_{E}\delta \A \cdot v_E^\perp - Dw_E\delta\A \cdot a^\perp_i- Da_{i}\delta \A \cdot v_E^\perp.
\end{align*} 
We have $Da_i\delta \A=\delta a_i$ and in view of \eqref{eq:29}, \eqref{eq:32}, \eqref{eq:36} we have that $Dv_{E}\delta \A=\F(i,v_E)$ and $Dw_{E}\delta \A=\F(i,w_E)$, thus
\begin{equation}
\nabla \gamma\cdot\delta \A = \text{sign}(J(v_E,w_E,\si_i))(\F(i,v_E)\cdot(-w^\perp_E + a^\perp_i) + \F(i,w_E)\cdot(v^\perp_E-a^\perp_i)-\delta a_i(-w^\perp_E+v^\perp_E)).
\label{eq:grad(a)}
\end{equation}
Now, we compute
\begin{align*}
\nabla \sigma \cdot \delta \A = [\nabla(\|v_E-a_i\|^2)+\nabla((v_E-a_i)\cdot(w_E-a_i))+\nabla(\|w_E-a_i\|^2)] \cdot \delta \A.
\end{align*}
We have
\begin{align*}
\nabla((v_E-a_i)\cdot(w_E-a_i))\cdot\delta\A&=[(w_E-a_i)\nabla(v_E-a_i)+(v_E-a_i)\nabla(w_E-a_i)]\cdot \delta \A \\
& =(w_E-a_i)\cdot[(Dv_E-Da_i)\cdot \delta \A] + (v_E-a_i)\cdot[(Dw_E-Da_i)\cdot\delta \A] \\
& =(w_E-a_i)\cdot(\F(i,v_E)-\delta a_i)+(v_E-a_i)\cdot (\F(i,w_E)-\delta a_i) \\
& =\F(i,v_E)\cdot(w_E-a_i)+\F(i,w_E)\cdot(v_E-a_i)-\delta a_i(w_E+v_E-2a_i),
\end{align*}
and in a similar way
\begin{align*}
\nabla(\|v_E-a_i\|^2)\cdot \delta \A
& = 2\F(i,v_E)\cdot (v_E-a_i) - 2\delta a_i\cdot (v_E-a_i)\\ 
\nabla(\|w_E-a_i\|^2) \cdot \delta \A
& = 2\F(i,w_E)\cdot (w_E-a_i) - 2\delta a_i\cdot (w_E-a_i).
\end{align*}
Combining these results, we get
\begin{align}
\nabla\sigma \cdot \delta \A &= 2\F(i,v_E)\cdot (v_E-a_i) - 2\delta a_i\cdot (v_E-a_i) + \F(i,v_E)\cdot(w_E-a_i)+\F(i,w_E)\cdot(v_E-a_i)\nonumber \\&\quad-\delta a_i(w_E+v_E-2a_i)
+2\F(i,w_E)\cdot (w_E-a_i) - 2\delta a_i\cdot (w_E-a_i) \nonumber \\
&=\F(i,v_E)\cdot(2v_E+w_E-3a_i)+\F(i,w_E)\cdot(v_E+2w_E-3a_i) \label{eq:grad(b)} 	\\ &\quad-3\delta a_i\cdot(v_E+w_E-2a_i) \nonumber.
\end{align}
Replacing \eqref{eq:grad(a)} and~\eqref{eq:grad(b)} in \eqref{eq:gradexpl}, we obtain
\begin{align*}
\nabla G_i(\A)\cdot \delta \A &= \frac{1}{12}\sum_{E\in \E_i}\kappa_i(E)(\F(i,v_E)\cdot(-w^\perp_E + a^\perp_i) + \F(i,w_E)\cdot(v^\perp_E-a^\perp_i)-\delta a_i(-w^\perp_E+v^\perp_E)) \\
& \quad + |J(v_E,w_E,\si_i)|(\F(i,v_E)\cdot(2v_E+w_E-3a_i)+\F(i,w_E)\cdot(v_E+2w_E-3a_i) \\ &\quad -3\delta a_i\cdot(v_E+w_E-2a_i)),
\end{align*}
where
\begin{equation*}
\kappa_i(E):=\text{sign}(J(v_E,w_E,\si_i))(\|v_E-a_i\|^2+(v_E-a_i)\cdot(w_E-a_i)+\|w_E-a_i\|^2).
\end{equation*}
Finally, rearranging the last expression for $\nabla G_i(\A)\cdot\delta \A$, we get 
\begin{equation}\label{grad2}
\nabla G(\A) \cdot \delta \A = \frac{1}{\kappa_0}\sum^{\kappa_0}_{i=1}\nabla G_i(\A)\cdot \delta \A,
\end{equation}
with
\begin{align} \label{grad3}
\begin{split}
\nabla G_i(\A)\cdot \delta \A &= \frac{1}{12}\sum_{E\in \E_i}\F(i,v_E)\cdot(\kappa_i (E)(-w_E^\perp + a^\perp_i) + |J(v_E,w_E,\si_i)|(2v_E+w_E-3a_i)) \\
& \hspace{1.5cm}  +\F(i,w_E)(\kappa_i (E)(v^\perp_E-a^\perp_i)+|J(v_E,w_E,\si_i)|(v_E+2w_E-3a_i) \\[2mm]
& \hspace{1.5cm} -\delta a_i \cdot (\kappa_i (E)(-w_E^\perp+v^\perp_E)+3|J(v_E,w_E,\si_i)|(v_E+w_E-2a_i) ).\end{split} 
\end{align}
Formula \eqref{grad2} is equivalent to \eqref{grad1} for the special case $\rho=1$. 
As explained above, formula~\eqref{grad3} illustrates how the gradient of objective functions depending on Voronoi diagrams can be computed directly, without having to differentiate a parameterized integral. However, one drawback is that this explicit calculation only works for specific densities $\rho$. Also, we expect this formula to be more computationally expensive than \eqref{grad1}, since \eqref{grad1} results from a simplification ($I^1_E$ and $I^2_E$ canceling out). This means that \eqref{grad2}, \eqref{grad3} could be further simplified for more efficiency, but this reduces the advantage of the explicit calculation.

\section{Numerical experiments}\label{sec:num}

In this section, we show numerical experiments related to the construction of centroidal Voronoi tessellations with special desired features. CVTs are always constructed by minimizing a combination of the energy function $G(\A)$ with an additional term that forces the desired geometric feature subject to $\A \in \R^{2\kappa_0}$ in the domain $A=[0,\sqrt{\kappa_0}]^2$, where $\kappa_0$ is the number of sites. In Section~\ref{CVT} we perform a numerical experiment in which we compare in practice the expressions \eqref{grad1} and \eqref{grad2} of the gradient of $G(\A)$. In Section \ref{CVT I.V}, we deal with the problem of constructing CVTs with cells of equal area. In Section~\ref{CVT A.S.E}, we show how to avoid small edges. In Section \ref{CVT S.C.V}, we deal with the construction of CVTs with cells of different sizes for different regions of the domain $A$.

The entire code is written in Fortran 90. Voronoi diagrams are computed with the implementation provided in~\cite{MR4267494,Birgin2022,birgin:01}. All experiments were performed on a computer with an Apple M1 processor and 8 GB of RAM, running MacOS Sonoma (version 14.6.1). The code was compiled using the GFortran compiler of GCC (version 14.1.0) with the -O3 optimization directive enabled.

\subsection{Centroidal Voronoi tessellation} \label{CVT}

In this section we minimize $G(\A)$, defined in \eqref{eq:mfun2}, without additional geometric terms. For general densities $\rho$, quadrature rules must be used to evaluate \eqref{eq:mfun2}. In all our experiments, we focus on the case $\rho \equiv 1$ and the explicit form \eqref{eq:explicitG} of $G(\A)$; this allows us to preserve the exactness of the function and to save computational time in the optimization process. The optimization problems are solved with the quasi-Newton method L-BFGS-B \cite{nocedal:02,nocedal:03,nocedal:01}. Default values are used for all its parameters. The stopping criterion is set to obtain an infinity norm of the continuous projected gradient less than or equal to $\epsilon=10^{-8}$, i.e., to find an iterate $\A^k$ such that
\[
\left\| P_A \left( \A^k - \nabla G(\A^k) \right) - \A^k \right\|_{\infty} \leq \epsilon = 10^{-8},
\]
where $P_A$ represents the (orthogonal) projection operator onto the (convex) feasible set $A=[0,\sqrt{\kappa_0}]^2$. The initial point $\A^0$ of the optimization process consists of uniformly distributed random points in the domain $A$. This type of starting point is only used in this experiment, where only the function $G(\A)$ is minimized. In Sections~\ref{CVT I.V}, \ref{CVT A.S.E}, and~\ref{CVT S.C.V}, where the combination of~$G(\A)$ with an extra term is minimized, we take as starting point~$\A^0$ the approximate solution~$\A^\star$ to the problem of minimizing~$G(\A)$.

Tables \ref{tab1} and \ref{tab2} show some details of the optimization process on a set of instances with $\kappa_0 \in \{5, 10, 50, 100, 1000, 10000, 25000, 50000\}$, considering expressions \eqref{grad1} and \eqref{grad2} for computing the gradient, respectively. In the tables, $\kappa_0$ is the number of sites considered. The columns $G(\A^\star)$ and $\|\nabla G(\A^\star)\|_{\infty}$ denote the value of the objective function and the sup-norm of the continuous projected gradient at the final iterate $\A^\star$. ``it'' is the number of iterations, ``fcnt'' is the number of evaluations of the objective function, and ``Time'' is the elapsed CPU time in seconds. The column ``fcnt/it'' shows the average number of function evaluations per iteration. From these tables we can see that, for each $\kappa_0$, the values of $G(\A^\star)$ are very similar when using both gradient expressions, and that in all cases it was possible to achieve the stopping criterion imposed. The tables also show that the performance of the method was slightly faster when using the expression \eqref{grad1} than when using the expression \eqref{grad2}. This confirms the observation of Section~\ref{subsec:2.4}, where it is explained that \eqref{grad2} is expected to be more computationally demanding than \eqref{grad1}. For this reason, in Sections \ref{CVT I.V}, \ref{CVT A.S.E}, and \ref{CVT S.C.V} the expression~\eqref{grad1} is used to compute the gradient of~$G$. Figure \ref{fig1} shows the resulting diagrams for the cases $\kappa_0 \in \{ 500, 1000 \}$.

\begin{table}[ht!]
\centering
\begin{tabular}{|c| c c c c c c|}
\hline
$\kappa_0$ & $G(\A^\star)$ & $\|\nabla G(\A^\star)\|_{\infty}$ & it & fcnt & Time & fcnt/it \\
\hline
\hline  
    5 & 1.76349E$-$01 & 6.8E$-$09 &   24 &   30 &   0.001 & 1.25\\
   10 & 1.69930E$-$01 & 8.6E$-$09 &   31 &   35 &   0.002 & 1.13\\ 
   50 & 1.65505E$-$01 & 9.6E$-$09 &   88 &   95 &   0.022 & 1.08\\
  100 & 1.63885E$-$01 & 9.9E$-$09 &  115 &  124 &   0.050 & 1.08\\
  500 & 1.62358E$-$01 & 9.5E$-$09 &  254 &  278 &   0.320 & 1.09\\
 1000 & 1.62019E$-$01 & 8.7E$-$09 &  477 &  530 &   1.208 & 1.11\\
10000 & 1.61780E$-$01 & 8.8E$-$09 &  494 &  505 &  12.915 & 1.02\\
25000 & 1.61708E$-$01 & 8.3E$-$09 &  753 &  780 &  60.151 & 1.04\\
50000 & 1.61661E$-$01 & 8.7E$-$09 & 1221 & 1271 & 472.639 & 1.04\\
\hline
\end{tabular}
\caption{Details of the optimization process and the solutions found for the problem of minimizing $G(\A)$ with increasing values of $\kappa_0$ and using \eqref{grad1} for computing $\nabla G$.} 
\label{tab1}
\end{table}

\begin{table}[ht!]
\centering
\begin{tabular}{|c| c c c c c c|}
\hline
$\kappa_0$ & $G(\A^\star)$ & $\|\nabla G(\A^\star)\|_{\infty}$ & it & fcnt & Time & fcnt/it \\
\hline
\hline  
    5 & 1.76349E$-$01 & 6.8E$-$09 &   24 &   30 &   0.002 & 1.25\\
   10 & 1.69930E$-$01 & 8.6E$-$09 &   31 &   35 &   0.004 & 1.13\\ 
   50 & 1.65443E$-$01 & 9.6E$-$09 &  107 &  120 &   0.034 & 1.12\\
  100 & 1.63885E$-$01 & 9.9E$-$09 &  115 &  124 &   0.061 & 1.08\\
  500 & 1.62241E$-$01 & 4.9E$-$09 &  239 &  258 &   0.376 & 1.08\\
 1000 & 1.62170E$-$01 & 8.5E$-$09 &  202 &  208 &   0.621 & 1.03\\
10000 & 1.61783E$-$01 & 9.1E$-$09 &  507 &  522 &  16.301 & 1.03\\
25000 & 1.61719E$-$01 & 8.8E$-$09 &  736 &  762 &  63.164 & 1.04\\
50000 & 1.61649E$-$01 & 8.6E$-$09 & 1285 & 1343 & 598.587 & 1.05\\
\hline
\end{tabular}
\caption{Details of the optimization process and the solutions found for the problem of minimizing $G(\A)$ with increasing values of $\kappa_0$ and using \eqref{grad2} for computing $\nabla G$.} 
\label{tab2}
\end{table}

\begin{figure}[ht!]
\begin{center}
\resizebox{\textwidth}{!}{
\begin{tabular}{cc}
\includegraphics[scale=2]{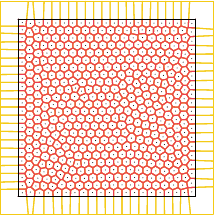} &
\includegraphics[scale=1.5]{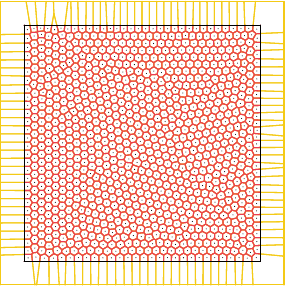}
\end{tabular}}
\end{center}
\caption{Centroidal Voronoi tessellations with $\kappa_0 \in \{500, 1000\}$. Results obtained using the gradient formula \eqref{grad1}.}
\label{fig1}
\end{figure}

\subsection{Centroidal Voronoi tessellation with cells of identical area.} \label{CVT I.V}


In this section, we consider the merit function given by
\[
f_1(\A)\coloneq \omega \, G(\A) + J^1(\A),
\]
where
\[
J^1(\A) := \frac{1}{\kappa_0}\sum^{\kappa_0}_{i=1}[J^1_i(\A)]^2 \mbox{ with }
J^1_i(\A):=\left(\int_{V_i(\A)}\,dx\right)/\left(\frac{1}{\kappa_0}\int_A \, dx \right)-1,
\]
and $\omega \geq 0$ is given. The function $J^1(\A)$ measures the deviation of the area of each Voronoi cell~$V_i(\A)$ with respect to the average area of the cells in the domain~$A$. The purpose of minimizing $f_1(\A)$ is to find CVTs with cells of similar area, i.e., what is expected in an approximate solution $\A$ is that for all $i$, $|V_i(\A)| \approx \frac{1}{\kappa_0}\int_A \, dx = 1$, because $A=[0,\sqrt{\kappa_0}]^2$. Since $J^1_i$ corresponds to \eqref{perturbedG1} with $f\equiv 1$, applying~\eqref{eq:cellint}, we have that
\[
\nabla J^1_i(\A) \cdot \delta \A = \frac{\kappa_0}{|A|}\sum_{E\in \E^{\text{int}}_i}\frac{|E|}{\| a_i-a_{k(i,E)} \|}[\delta a_i \cdot (p_E-a_i)-\delta a_{k(i,E)}\cdot (p_E-a_{k(i,E)})],
\]
where $p_E\coloneq (v_E+w_E)/2$ and $k(i,E)$ is the index such that $E=\overline{V_i(\A)}\cap \overline{V_{k(i,E)}(\A)}$.

The choice of the parameter $\omega$ is important for the optimization process to obtain the desired results. We will present an appropriate choice for $\omega$ that was found in the particular case where $\kappa_0=10$ is used, and then show the results obtained when optimizing the cases with $\kappa_0\in \{ 500, 1000 \}$ using the same value for $\omega$. Figure~\ref{fig2}(a) shows the solution of minimizing $G(\A)$, while Figures~\ref{fig2}(b-e) show the solutions obtained by minimizing $f_1(\A)$ with $\omega \in \{ 1, 0.1, 0.01, 0.001 \}$. In the figures, the cells $V_i(\A)$ that satisfy $|J^1_i(\A)|>10^{-3}$ are colored green. Therefore, the uncolored cells $V_i$ satisfy $|J^1_i(\A)| \leq 10^{-3}$, which means that the uncolored cells have an area that is very close to the desired area (the relative error of the cell area to the ideal area is less than or equal to 0.1\%). The desired goal is reached when all cells are uncolored, which is the case for $\omega=0.001$. Table~\ref{tab3} shows details of the solutions found. In particular, it shows the area of each cell. It is interesting to compare the cell areas of the diagrams constructed by minimizing $G$ alone and $f_1$ with $\omega=0.001$. Table \ref{tab4} shows some details of the optimization process. In the table, the columns $f_1(\A^\star)$ and $\|f_1(\A^\star)\|_\infty$ indicate the value of the objective function and the sup-norm of the continuous projected gradient at the final iterate $\A^\star$. The columns $G(\A^\star)$ and $J^1(\A^\star)$ identify the value of the CVT energy function and the function $J^1$ at the final iterate $\A^\star$. The remaining columns contain the number of iterations, the number of function evaluations, and the CPU time in seconds. The numbers in the table show that, regardless of the value of $\omega$, the problems were easily solved, with the optimization process stopping in all cases due to the imposed stopping criterion. The value of $G$ in the solution $\A^\star$ obtained by minimizing $G$ alone is $G(\A^\star) = $ 1.69930E-01; see Table~\ref{tab2}. As $\omega$ decreases, Table~\ref{tab4} shows that $J^1(\A^\star)$ improves by at least an order of magnitude, while $G(\A^\star)$ deteriorates only slightly, remaining close to the value obtained by minimizing $G$ alone. This means that minimizing $f_1$ succeeds in designing a diagram with the desired geometric properties at the cost of only a small increase in the CVT energy function, which is the function that a CVT diagram should minimize. In short, minimizing $f_1$ succeeds in designing a CVT diagram with the desired characteristics.

\begin{figure}[ht!] 
\begin{subfigure}{1 \linewidth}
\centering
\includegraphics[scale=2.25]{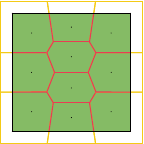}
\caption{ $G(\A)$ }
\end{subfigure}
\begin{subfigure}{0.5 \linewidth}
\centering
\includegraphics[scale=2.25]{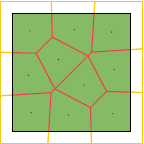}
\caption{$f(\A)$ with $\omega = 1$}
\end{subfigure}
\begin{subfigure}{0.5 \linewidth}
\centering
\includegraphics[scale=2.25]{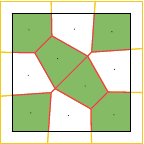}
\caption{$f(\A)$ with $\omega = 0.1$}
\end{subfigure}
\begin{subfigure}{0.5 \linewidth}
\centering
\includegraphics[scale=2.25]{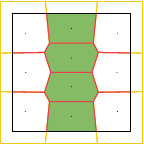}
\caption{$f(\A)$ with $\omega = 0.01$}
\end{subfigure}
\begin{subfigure}{0.5 \linewidth}
\centering
\includegraphics[scale=2.25]{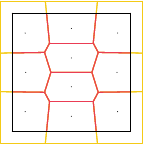}
\caption{$f(\A)$ with $\omega = 0.001$}
\end{subfigure}
\caption{Centroidal Voronoi tessellations with $\kappa_0=10$. In (a) we show the result of minimizing the function $G(\A)$. In (b-e) we show the diagrams obtained by minimizing $f_1(\A) = \omega \, G(\A)+J^1(\A)$ with decreasing values of $\omega$.}
\label{fig2}
\end{figure}

\begin{table}[ht!]
\centering
\resizebox{\textwidth}{!}{
\begin{tabular}{|c|c|c|c|c|c|c|c|c|c|c|}
\hline
Cell 
& \multicolumn{2}{|c|}{$G(\A)$}
& \multicolumn{2}{|c|}{$f_1(\A)$ with $\omega=1$} 
& \multicolumn{2}{|c|}{$f_1(\A)$ with $\omega=0.1$} 
& \multicolumn{2}{|c|}{$f_1(\A)$ with $\omega=0.01$} 
& \multicolumn{2}{|c|}{$f_1(\A)$ with $\omega=0.001$} \\
\hline
$i $ 
& $|V_i(\A^\star)|$ & $|J^1_i(\A^\star)|$ 
& $|V_i(\A^\star)|$ & $|J^1_i(\A^\star)|$ 
& $|V_i(\A^\star)|$ & $|J^1_i(\A^\star)|$ 
& $|V_i(\A^\star)|$ & $|J^1_i(\A^\star)|$ 
& $|V_i(\A^\star)|$ & $|J^1_i(\A^\star)|$ \\
\hline
\hline
1  & 8.31E$-$01 & 1.69E$-$01 & 9.97E$-$01 & 3.28E$-$03 & 9.99E$-$01 & 7.09E$-$04 & 9.98E$-$01 & 1.58E$-$03 & 1.00E$+$00 & 1.60E$-$04\\
2  & 1.08E$+$00 & 8.10E$-$02 & 9.97E$-$01 & 3.28E$-$03 & 9.99E$-$01 & 7.09E$-$04 & 1.00E$+$00 & 9.40E$-$04 & 1.00E$+$00 & 9.50E$-$05\\
3  & 1.09E$+$00 & 8.75E$-$02 & 1.04E$+$00 & 4.08E$-$02 & 1.01E$+$00 & 6.30E$-$03 & 1.00E$+$00 & 9.86E$-$04 & 1.00E$+$00 & 9.95E$-$05\\
4  & 1.09E$+$00 & 8.75E$-$02 & 1.04E$+$00 & 4.08E$-$02 & 1.01E$+$00 & 6.30E$-$03 & 1.00E$+$00 & 9.86E$-$04 & 1.00E$+$00 & 9.95E$-$05\\
5  & 9.13E$-$01 & 8.72E$-$02 & 9.55E$-$01 & 4.47E$-$02 & 9.94E$-$01 & 6.43E$-$03 & 9.99E$-$01 & 1.33E$-$03 & 1.00E$+$00 & 1.34E$-$04\\
6  & 1.09E$+$00 & 8.75E$-$02 & 1.01E$+$00 & 1.04E$-$02 & 1.00E$+$00 & 1.55E$-$03 & 1.00E$+$00 & 9.86E$-$04 & 1.00E$+$00 & 9.95E$-$05\\
7  & 9.13E$-$01 & 8.72E$-$02 & 9.55E$-$01 & 4.47E$-$02 & 9.94E$-$01 & 6.43E$-$03 & 9.99E$-$01 & 1.33E$-$03 & 1.00E$+$00 & 1.34E$-$04\\
8  & 1.08E$+$00 & 8.10E$-$02 & 9.97E$-$01 & 3.28E$-$03 & 9.99E$-$01 & 7.09E$-$04 & 1.00E$+$00 & 9.40E$-$04 & 1.00E$+$00 & 9.50E$-$05\\
9  & 8.31E$-$01 & 1.69E$-$01 & 9.97E$-$01 & 3.28E$-$03 & 9.99E$-$01 & 7.09E$-$04 & 9.98E$-$01 & 1.58E$-$03 & 1.00E$+$00 & 1.60E$-$04\\
10 & 1.09E$+$00 & 8.75E$-$02 & 1.01E$+$00 & 1.04E$-$02 & 1.00E$+$00 & 1.55E$-$03 & 1.00E$+$00 & 9.86E$-$04 & 1.00E$+$00 & 9.95E$-$05\\
\hline
\end{tabular}}
\caption{Obtained values for the area $|V_i(\A^\star)|$ of each cell and $J^1(\A^\star)$ after minimizing the functions $G(\A)$ and $f_1(\A) = \omega \, G(\A) + J^1(\A)$ with $\omega \in \{1, 0.1, 0.01, 0.001\}$.}
\label{tab3}
\end{table}

\begin{table}[ht!]
\centering
\begin{tabular}{|cccccccc|}
\hline
$\omega$ & $f_1(\A^\star)$ & $\|\nabla f_1(\A^\star)\|_{\infty}$ & $G(\A^\star)$ & $J^1(\A^\star)$ & it & fcnt & Time  \\
\hline
\hline
1     & 1.73112E$-$01 & 1.9E$-$09 & 1.72354E$-$01 & 7.57491E$-$04 & 44 &  63 & 0.004\\
0.1   & 1.74133E$-$02 & 9.1E$-$09 & 1.73964E$-$01 & 1.68787E$-$05 & 94 & 119 & 0.005\\
0.01  & 1.82952E$-$03 & 9.7E$-$09 & 1.82810E$-$01 & 1.41961E$-$06 & 29 &  32 & 0.002\\
0.001 & 1.83081E$-$04 & 2.7E$-$09 & 1.83066E$-$01 & 1.44641E$-$08 & 35 &  38 & 0.002\\
\hline
\end{tabular}
\caption{Details of the process of minimizing the function $f_1(\A) = \omega \, G(\A)+ J^1(\A)$ varying~$\omega$.}
\label{tab4}
\end{table}

Taking into account $\kappa_0=1000$, we also perform the experiment of minimizing only $G$ and $f_1$ with $\omega \in \{ 1, 0.1, 0.01, 0.001 \}$. Figure \ref{fig3} shows the distribution of cell areas in the four different approximate solutions $\A^\star(\omega)$ found by minimizing the function $f_1(\A)$ varying $\omega$. Consider the solutions $\A^\star(\omega)$ for the different values of $\omega$ and let $\nu_1 = \min_{\{\omega,i\}} |V_i(\A^\star(\omega))|$, $\nu_{101} = \max_{\{\omega,i\}} |V_i(\A^\star(\omega))|$, $\Delta \nu = (\nu_{101} - \nu_1)/100$, and $\nu_j = \nu_1 + j \, \Delta \nu$ for $j=2,\dots,100$. The figure shows the distribution of the values of $|V_i(\A^\star(\omega))|$ for $\omega\in \{ 1, 0.1, 0.01, 0.001 \}$ and $i=1,\dots,\kappa_0$ over the intervals $[\nu_j,\nu_{j+1}]$ for $j=1,\dots,100$. More specifically, there is a graph for each $\A^\star(\omega)$ and, for each value $\frac{1}{2}(\nu_j+\nu_{j+1})$ in the abscissa, the graph shows in the ordinate the proportion of cells $V_i(\A^\star(\omega))$ whose area $|V_i(\A^\star(\omega))|$ is in the interval $[\nu_j,\nu_{j+1}]$. The plot clearly shows that, the smaller the value of $\omega$, the larger the proportion of cells in the solution $\A^\star(\omega)$ whose area is close to~1. Figure~\ref{fig4}(a-b) shows the diagrams obtained by minimizing $G$ alone and $f_1$ with $\omega = 0.001$, respectively. In both figures, cells $V_i(\A)$ such that $|J^1_i(\A)|>10^{-3}$ are painted green. Basically speaking, this means that when minimizing $G$ alone, almost none of the cells have the desired area, while when minimizing $f_1$ with $\omega=0.001$, all have the desired area.

\begin{figure}[ht!]
\centering
\includegraphics[width=0.95\linewidth]{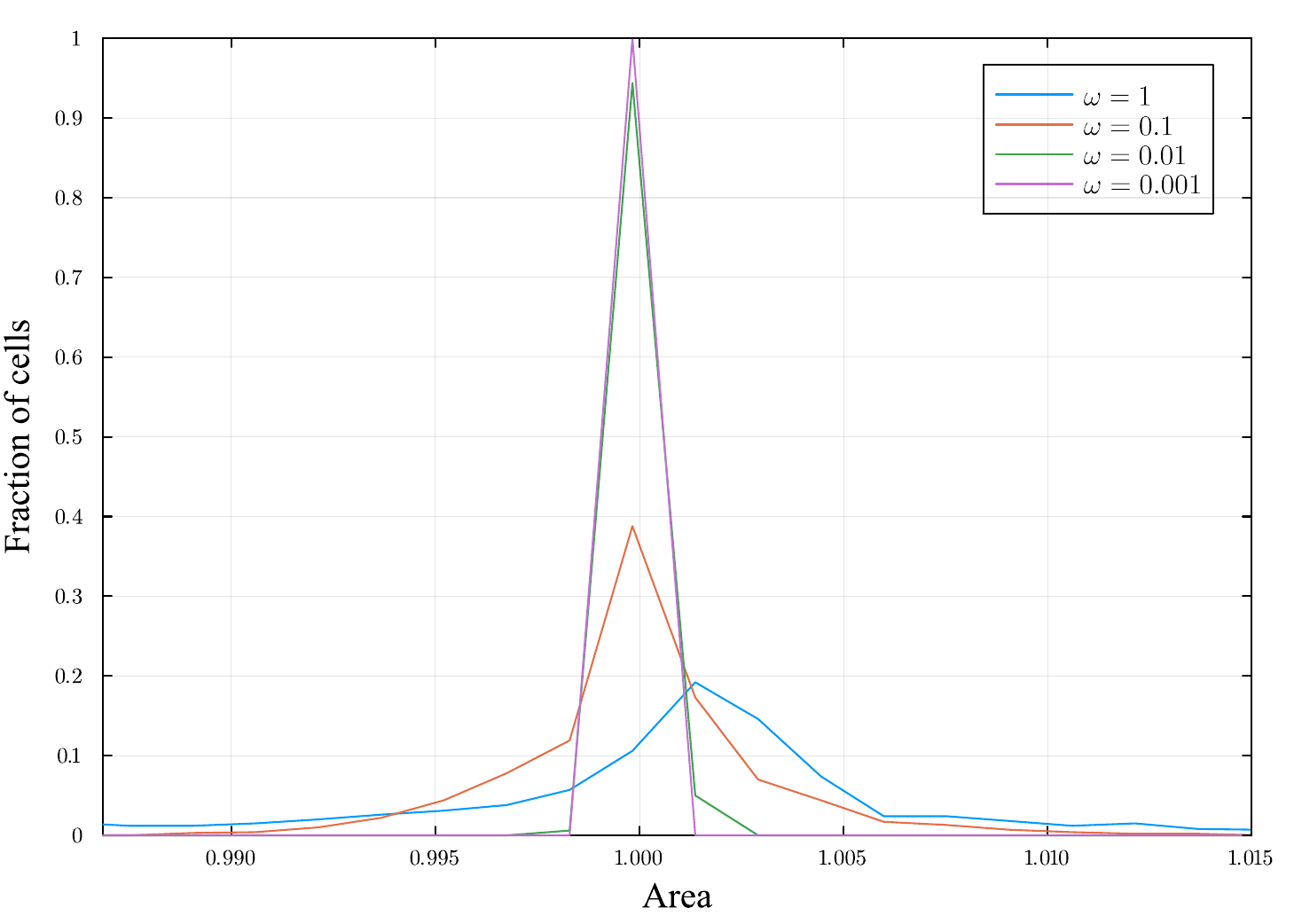}
\caption{Distribution of the cell areas in the approximate solutions found when minimizing the function $f_1(\A) = \omega \, G(\A)+J^1(\A)$ with $\kappa_0=1000$ considering $\omega\in \{ 1, 0.1, 0.01, 0.001 \}$.}
\label{fig3}
\end{figure}

\begin{figure}[ht!]
\begin{subfigure}{0.5 \linewidth}
\centering
\includegraphics[scale=1.6]{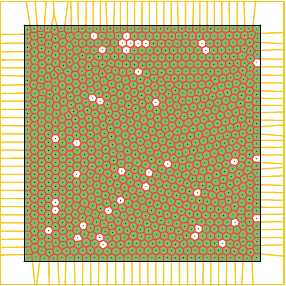}
\caption{$G(\A)$}
\end{subfigure}
\begin{subfigure}{0.5\linewidth}
\centering
\includegraphics[scale=1.6]{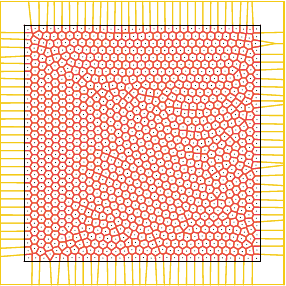}
\caption{$f_1(\A)$ with $\omega=0.001$} 
\end{subfigure}
\caption{Centroidal Voronoi tessellation with $\kappa_0=1000$. In (a) we show the diagram resulting from minimizing the function $G(\A)$. In (b) we show the diagrams obtained by minimizing $f_1(\A) = \omega \, G(\A) + J^1(\A)$ with $\omega=0.001$.}
\label{fig4}
\end{figure}

\subsection{Centroidal Voronoi tessellation avoiding cells with small edges} \label{CVT A.S.E}

In this section, we consider the size of the edges of the cells. Scrutinizing the cells in Figure~\ref{fig4}(a), we observe the presence of cells with small edges. Specifically, given a fraction $c_2 \in (0,1)$, we say that ``an edge $E$ of a cell $V_i(\A)$ is small'' if its size $|E|$ is smaller than $c_2 \bar E_i$, where $\bar E_i = P_i / n_i$ is the average of the edge sizes of the cell $V_i(\A)$, $P_i$ is the perimeter of the cell, and $n_i$ is the number of edges of the cell. Given a tolerance $c_2 \in (0,1)$, to construct CVTs that do not have cells with small edges, we consider the objective function
\[
f_2(\A) := \omega \, G(\A) + J^2(\A),
\]
where 
\begin{equation}\label{eq:GJ2b}
J^2(\A) := \sum^{\kappa_0}_{i=1}J^2_i(\A) \text{ with } J^2_i(\A) := \frac{1}{n_i}\sum_{E\in \E_i}\min\left\{ 0, \frac{|E|}{\bar{E}_i}-c_2 \right \}^2
\end{equation}
and $\omega \geq 0$ is a given constant. In~\eqref{eq:GJ2b}, $\E_i$ is the set of edges of cell $V_i(\A)$ and $n_i = |\E_i|$. Given $c_2 \in (0,1)$, if all edges $E\in \E_i$ of a cell $V_i(\A)$ satisfy $|E| \geq c_2 \bar{E}_i$, i.e., if the size of each cell edge is at least $100\% \times c_2$ of the average size, then $|E|/\bar{E}_i - c_2 \geq 0$ for all $E$, and hence $J^2_i$ vanishes. Thus, $J^2_i$ measures the violation of the constraints on the minimum size of the edges of the cell $V_i(\A)$. The gradient $\nabla J^2_i$ depends on $\nabla |E|$ and $|E|$ corresponds to \eqref{eq:perturbedG2} with $f \equiv 1$. Therefore, by (\ref{eq:35},\ref{eq:36}), we have that
\[
\nabla J^2_i(\A) \cdot \delta \A =  \sum_{E\in\E_i}\mu(E)([\F(i,w_E) - \F(i,v_E)] \cdot \tau_E)
\]
with
\[
\mu(E):= \frac{2}{P_i}\left( \min\left\lbrace 0, \frac{|E|}{\bar{E}_i}-c_2 \right\rbrace-\sum_{\tilde{E}\in \E_i}\frac{|\tilde{E}|}{P_i}\min\left\lbrace 0, \frac{|\tilde{E}|}{\bar{E}_i}-c_2 \right\rbrace\right).
\]

To obtain an appropriate value for $\omega$, we analyze the results of minimizing $f_2(\A)$ with $c_2=0.5$ and $\kappa_0=10$, varying $\omega \in \{1,0.1,0.01\}$. Table \ref{tab5} shows the results. The numbers in the table show that all the problems were easily solved and stopped by the imposed criterion. Furthermore, we can see that, as expected, the smaller the value of $\omega$, the smaller the value of $J^2(\A^\star(\omega))$. We can also see that as $\omega$ decreases, the value of $G$ increases. However, it increases only slightly and always remains close to the value obtained by minimizing the function $G$ alone, which is 1.69930E-01 (see Table~\ref{tab2}). This suggests that the solutions found, which satisfy the desired geometric properties, also preserve the property of minimizing the CVT energy function. Figures~\ref{fig5}(a-d) show the diagrams obtained when minimizing $G(\A)$ and $f_2(\A)$ with $c_2=0.5$ and varying $\omega \in \{1,0.1,0.01\}$. Considering $c_2=0.5$ means that we expect diagrams in which no cell has an edge smaller than 50\% of the average size of its edges. In the figures, colored cells are those that do not satisfy the desired property. This means that satisfactory results were obtained when minimizing $f_2(\A)$ with $\omega=0.1$ and $\omega=0.01$. Table~\ref{tab6} analyzes the obtained solutions in detail. For each cell of each solution, the table shows the size of its edges, the average size of the edges, the ratio of the smallest edge to the average, and the value of $J_i^2$. The table clearly shows that as $\omega$ decreases, the smallest edge of each cell approaches the smallest required proportion, that is, 50\%.

\begin{table}[ht!]
\centering
\begin{tabular}{|c c c c c c c c|}
\hline
$\omega$ & $f_2(\A^\star)$ & $\| \nabla f_2(\A^\star) \|_{\infty}$ & $G(\A^\star)$ & $J^2(\A^\star)$ & it & fcnt & Time \\
\hline
\hline
1    & 1.70217E$-$01 & 1.9E$-$09 & 1.70079E$-$01 & 1.37640E$-$04 &  8 &  9 & 0.002\\
0.1  & 1.70437E$-$02 & 7.3E$-$09 & 1.70393E$-$01 & 4.34693E$-$06 & 17 & 21 & 0.002\\
0.01 & 1.70479E$-$03 & 9.9E$-$09 & 1.70474E$-$01 & 5.12211E$-$08 & 27 & 30 & 0.003\\
\hline
\end{tabular}
\caption{Details of the process of minimizing $f_2(\A) = \omega \, G(\A) + J^2(\A)$ with $c_2=0.5$ varying $\omega$.}
\label{tab5}
\end{table}

\begin{table}[ht!]
\centering
\resizebox{\textwidth}{!}{
\begin{tabular}{ |c|c| c c c c c c | c c c|}
\hline
& Cell $i$ & \multicolumn{6}{c|}{Edges sizes} & $\bar{E}_i$ & $100(|E^{\min}_i| / \bar{E}_i )\%$ & $J^2_i(\A^\star)$ \\
\hline
\hline
\multirow{10}{*}{\rotatebox{90}{Min $G(\A)$}}
&1  & 1.20E$+$00 & 7.75E$-$01 & 9.70E$-$01 & 7.75E$-$01 & -          & -          & 9.30E$-$01 & 83.4\%       & 0.00E$+$00\\
&2  & 1.05E$+$00 & 9.24E$-$01 & 5.65E$-$01 & 5.65E$-$01 & 9.24E$-$01 & -          & 8.06E$-$01 & 70.0\%       & 0.00E$+$00\\
&3  & 3.36E$-$01 & 7.75E$-$01 & 9.82E$-$01 & 1.05E$+$00 & 9.24E$-$01 & -          & 8.14E$-$01 & \pmb{41.3\%} & 1.51E$-$03\\
&4  & 9.82E$-$01 & 1.05E$+$00 & 9.24E$-$01 & 3.36E$-$01 & 7.75E$-$01 & -          & 8.14E$-$01 & \pmb{41.3\%} & 1.51E$-$03\\
&5  & 3.36E$-$01 & 9.70E$-$01 & 3.36E$-$01 & 5.65E$-$01 & 9.02E$-$01 & 5.65E$-$01 & 6.12E$-$01 & 54.9\%       & 0.00E$+$00\\
&6  & 9.24E$-$01 & 1.05E$+$00 & 9.82E$-$01 & 7.75E$-$01 & 3.36E$-$01 & -          & 8.14E$-$01 & \pmb{41.3\%} & 1.51E$-$03\\
&7  & 9.02E$-$01 & 5.65E$-$01 & 3.36E$-$01 & 9.70E$-$01 & 3.36E$-$01 & 5.65E$-$01 & 6.12E$-$01 & 54.9\%       & 0.00E$+$00\\
&8  & 5.65E$-$01 & 9.24E$-$01 & 1.05E$+$00 & 9.24E$-$01 & 5.65E$-$01 & -          & 8.06E$-$01 & 70.0\%       & 0.00E$+$00\\
&9  & 7.75E$-$01 & 1.20E$+$00 & 7.75E$-$01 & 9.70E$-$01 & -          & -          & 9.30E$-$01 & 83.4\%       & 0.00E$+$00\\
&10 & 1.05E$+$00 & 9.82E$-$01 & 7.75E$-$01 & 3.36E$-$01 & 9.24E$-$01 & -          & 8.14E$-$01 & \pmb{41.3\%} & 1.51E$-$03\\
\hline
\multicolumn{11}{c}{}\\
\hline
& Cell $i$ & \multicolumn{6}{c|}{Edges sizes} & $\bar{E}_i$ & $100(|E^{\min}_i| / \bar{E}_i )\%$ & $J^2_i(\A^\star)$    \\
\hline
\hline
\multirow{10}{*}{\rotatebox{90}{Min $f_2(\A)$ with $\omega=1$}}
&1  & 1.22E$+$00 & 7.54E$-$01 & 9.47E$-$01 & 7.54E$-$01 & -          & -          & 9.18E$-$01 & 82.2\%       & 0.00E$+$00\\
&2  & 1.07E$+$00 & 9.22E$-$01 & 5.56E$-$01 & 5.56E$-$01 & 9.22E$-$01 & -          & 8.05E$-$01 & 69.1\%       & 0.00E$+$00\\
&3  & 3.73E$-$01 & 7.54E$-$01 & 9.73E$-$01 & 1.05E$+$00 & 9.22E$-$01 & -          & 8.14E$-$01 & \pmb{45.9\%} & 3.44E$-$04\\
&4  & 9.73E$-$01 & 1.05E$+$00 & 9.22E$-$01 & 3.73E$-$01 & 7.54E$-$01 & -          & 8.14E$-$01 & \pmb{45.9\%} & 3.44E$-$04\\
&5  & 3.73E$-$01 & 9.47E$-$01 & 3.73E$-$01 & 5.56E$-$01 & 9.01E$-$01 & 5.56E$-$01 & 6.18E$-$01 & 60.4\%       & 0.00E$+$00\\
&6  & 9.22E$-$01 & 1.05E$+$00 & 9.73E$-$01 & 7.54E$-$01 & 3.73E$-$01 & -          & 8.14E$-$01 & \pmb{45.9\%} & 3.44E$-$04\\
&7  & 9.01E$-$01 & 5.56E$-$01 & 3.73E$-$01 & 9.47E$-$01 & 3.73E$-$01 & 5.56E$-$01 & 6.18E$-$01 & 60.4\%       & 0.00E$+$00\\
&8  & 5.56E$-$01 & 9.22E$-$01 & 1.07E$+$00 & 9.22E$-$01 & 5.56E$-$01 & -          & 8.05E$-$01 & 69.1\%       & 0.00E$+$00\\
&9  & 7.54E$-$01 & 1.22E$+$00 & 7.54E$-$01 & 9.47E$-$01 & -          & -          & 9.18E$-$01 & 82.2\%       & 0.00E$+$00\\
&10 & 1.05E$+$00 & 9.73E$-$01 & 7.54E$-$01 & 3.73E$-$01 & 9.22E$-$01 & -          & 8.14E$-$01 & \pmb{45.9\%} & 3.44E$-$04\\
\hline
\multicolumn{11}{c}{}\\
\hline
& Cell $i$ & \multicolumn{6}{c|}{Edges sizes} & $\bar{E}_i$ & $100(|E^{\min}_i| / \bar{E}_i )\%$ & $J^2_i(\A^\star)$  \\
\hline
\hline
\multirow{10}{*}{\rotatebox{90}{Min $f_2(\A)$ with $\omega=0.1$}}
&1  & 1.23E$+$00 & 7.40E$-$01 & 9.29E$-$01 & 7.40E$-$01 & -          & -          & 9.10E$-$01 & 81.3\%       & 0.00E$+$00\\
&2  & 1.08E$+$00 & 9.20E$-$01 & 5.50E$-$01 & 5.50E$-$01 & 9.20E$-$01 & -          & 8.04E$-$01 & 68.4\%       & 0.00E$+$00\\
&3  & 4.01E$-$01 & 7.40E$-$01 & 9.66E$-$01 & 1.04E$+$00 & 9.20E$-$01 & -          & 8.14E$-$01 & \pmb{49.3\%} & 1.09E$-$05\\
&4  & 9.66E$-$01 & 1.04E$+$00 & 9.20E$-$01 & 4.01E$-$01 & 7.40E$-$01 & -          & 8.14E$-$01 & \pmb{49.3\%} & 1.09E$-$05\\
&5  & 4.01E$-$01 & 9.29E$-$01 & 4.01E$-$01 & 5.50E$-$01 & 9.02E$-$01 & 5.50E$-$01 & 6.22E$-$01 & 64.4\%       & 0.00E$+$00\\
&6  & 9.20E$-$01 & 1.04E$+$00 & 9.66E$-$01 & 7.40E$-$01 & 4.01E$-$01 & -          & 8.14E$-$01 & \pmb{49.3\%} & 1.09E$-$05\\
&7  & 9.02E$-$01 & 5.50E$-$01 & 4.01E$-$01 & 9.29E$-$01 & 4.01E$-$01 & 5.50E$-$01 & 6.22E$-$01 & 64.4\%       & 0.00E$+$00\\
&8  & 5.50E$-$01 & 9.20E$-$01 & 1.08E$+$00 & 9.20E$-$01 & 5.50E$-$01 & -          & 8.04E$-$01 & 68.4\%       & 0.00E$+$00\\
&9  & 7.40E$-$01 & 1.23E$+$00 & 7.40E$-$01 & 9.29E$-$01 & -          & -          & 9.10E$-$01 & 81.3\%       & 0.00E$+$00\\
&10 & 1.04E$+$00 & 9.66E$-$01 & 7.40E$-$01 & 4.01E$-$01 & 9.20E$-$01 & -          & 8.14E$-$01 & \pmb{49.3\%} & 1.09E$-$05\\
\hline
\multicolumn{11}{c}{}\\
\hline
&Cell & \multicolumn{6}{c|}{Edges sizes} & $\bar{E}_i$ & $100(|E^{\min}_i| / \bar{E}_i )\%$ & $J^2_i(\A^\star)$  \\
\hline
\hline
\multirow{10}{*}{\rotatebox{90}{Min $f_2(\A)$ with $\omega=0.01$}}
&1  & 1.23E$+$00 & 7.37E$-$01 & 9.26E$-$01 & 7.37E$-$01 & -          & -          & 9.08E$-$01 & 81.1\%       & 0.00E$+$00\\
&2  & 1.08E$+$00 & 9.20E$-$01 & 5.48E$-$01 & 5.48E$-$01 & 9.20E$-$01 & -          & 8.04E$-$01 & 68.2\%       & 0.00E$+$00\\
&3  & 4.06E$-$01 & 7.37E$-$01 & 9.65E$-$01 & 1.04E$+$00 & 9.20E$-$01 & -          & 8.14E$-$01 & \pmb{49.9\%} & 1.28E$-$07\\
&4  & 9.65E$-$01 & 1.04E$+$00 & 9.20E$-$01 & 4.06E$-$01 & 7.37E$-$01 & -          & 8.14E$-$01 & \pmb{49.9\%} & 1.28E$-$07\\
&5  & 4.06E$-$01 & 9.26E$-$01 & 4.06E$-$01 & 5.48E$-$01 & 9.02E$-$01 & 5.48E$-$01 & 6.23E$-$01 & 65.2\%       & 0.00E$+$00\\
&6  & 9.20E$-$01 & 1.04E$+$00 & 9.65E$-$01 & 7.37E$-$01 & 4.06E$-$01 & -          & 8.14E$-$01 & \pmb{49.9\%} & 1.28E$-$07\\
&7  & 9.02E$-$01 & 5.48E$-$01 & 4.06E$-$01 & 9.26E$-$01 & 4.06E$-$01 & 5.48E$-$01 & 6.23E$-$01 & 65.2\%       & 0.00E$+$00\\
&8  & 5.48E$-$01 & 9.20E$-$01 & 1.08E$+$00 & 9.20E$-$01 & 5.48E$-$01 & -          & 8.04E$-$01 & 68.2\%       & 0.00E$+$00\\
&9  & 7.37E$-$01 & 1.23E$+$00 & 7.37E$-$01 & 9.26E$-$01 & -          & -          & 9.08E$-$01 & 81.1\%       & 0.00E$+$00\\
&10 & 1.04E$+$00 & 9.65E$-$01 & 7.37E$-$01 & 4.06E$-$01 & 9.20E$-$01 & -          & 8.14E$-$01 & \pmb{49.9\%} & 1.28E$-$07\\
\hline
\end{tabular}}
\caption{Details of cells after minimizing the objective functions $G(\A)$ and $f_2(\A)$ with $\omega \in \{1,0.1,0.01\}$.}
\label{tab6}
\end{table}

Upon determining the value $\omega=0.01$, additional experiments with $\kappa_0 \in \{ 500, 1000 \}$ were performed. The experiments consisted of minimizing $G(\A)$ as well as $f_2(\A)$ with $\omega=0.01$, varying $c_2$. Tables~\ref{tab7} and~\ref{tab8} and Figures~\ref{fig6} and~\ref{fig7} show the results. The results in the tables show that, when the value of $c_2$ increases, the optimization process is slightly more expensive. However, in all cases, the problems were easily solved. The interesting observation is that when $c_2$ increases, solutions $\A^\star$ with more restrictive geometric conditions are calculated with a very small increase in the value of $G(\A^\star)$. Moreover, the values of $G$ remain close to the value obtained when minimizing $G$ alone (see Table~\ref{tab2}). We can also see that as $c_2$ increases, $J^2(\A^{\star})$ also increases and, at $c_2=0.9$, the geometric condition is not met. In the figures, seven shades of blue were used to paint the cells. The darkest cells are those with side sizes between 10-20\% of the mean. The lighter cells are those with side sizes between 70 and 80\% of the mean. The figures show that, as the value of $c_2$ increases, cells with edge sizes smaller than 80\% of the average size disappear.

\begin{table}[ht!]
\centering
\begin{tabular}{|cccccccc|}
\hline
$c_2$ & $f_2(\A^\star)$ & $\| \nabla f_2(\A^\star) \|_{\infty}$ & $G(\A^\star)$ & $J^2(\A^\star)$ & it & fcnt & Time\\
\hline
\hline
0.2 & 1.62359E$-$03 & 6.4E$-$09 & 1.62359E$-$01 & 1.36588E$-$11 &  64 & 101 & 0.141\\
0.3 & 1.62361E$-$03 & 8.4E$-$09 & 1.62361E$-$01 & 6.44011E$-$11 & 114 & 145 & 0.186\\
0.4 & 1.62368E$-$03 & 9.2E$-$09 & 1.62368E$-$01 & 2.48965E$-$10 & 233 & 311 & 0.350\\
0.5 & 1.62384E$-$03 & 9.3E$-$09 & 1.62384E$-$01 & 8.14196E$-$10 & 335 & 454 & 0.498\\
0.6 & 1.62424E$-$03 & 6.1E$-$09 & 1.62423E$-$01 & 2.89568E$-$09 & 345 & 427 & 0.480\\
0.7 & 1.62564E$-$03 & 8.6E$-$09 & 1.62562E$-$01 & 1.46599E$-$08 & 507 & 558 & 0.638\\
0.8 & 1.63419E$-$03 & 8.6E$-$09 & 1.63398E$-$01 & 2.15309E$-$07 & 633 & 642 & 0.781\\
0.9 & 1.74145E$-$03 & 9.0E$-$09 & 1.69973E$-$01 & 4.17186E$-$05 & 722 & 753 & 0.997\\
\hline
\end{tabular}
\caption{Details of the optimization process and the solutions found for the problem of finding centroidal Voronoi tessellations that avoid cells with relatively small edges for $\kappa_0=500$.}
\label{tab7}
\end{table}

\begin{table}[ht!]
\centering
\begin{tabular}{|c c c c c c c c|}
\hline
$c_2$ & $f_2(\A^\star)$ & $\| \nabla f_2(\A^\star) \|_{\infty}$ & $G(\A^\star)$ & $J^2(\A^\star)$ & it & fcnt & Time \\
\hline
\hline
0.3 & 1.62019E$-$03 & 9.3E$-$09 & 1.62019E$-$01 & 3.01081E$-$12 &   48 &  63  & 0.179\\
0.4 & 1.62022E$-$03 & 9.8E$-$09 & 1.62022E$-$01 & 7.05410E$-$11 &   81 &  110 & 0.284\\
0.5 & 1.62032E$-$03 & 6.9E$-$09 & 1.62032E$-$01 & 3.57745E$-$10 &  235 &  300 & 0.711\\
0.6 & 1.62063E$-$03 & 4.6E$-$09 & 1.62063E$-$01 & 1.33134E$-$09 &  310 &  387 & 0.918\\
0.7 & 1.62174E$-$03 & 8.7E$-$09 & 1.62173E$-$01 & 8.84976E$-$09 &  475 &  512 & 1.235\\
0.8 & 1.62862E$-$03 & 9.1E$-$09 & 1.62850E$-$01 & 1.23288E$-$07 &  777 &  787 & 1.970\\
0.9 & 1.71539E$-$03 & 9.5E$-$09 & 1.67819E$-$01 & 3.72017E$-$05 & 1357 & 1450 & 4.031\\
\hline
\end{tabular}
\caption{Details of the optimization process and the solutions found for the problem of finding centroidal Voronoi tessellations that avoid cells with relatively small edges for $\kappa_0=1000$.}
\label{tab8}
\end{table}

\begin{figure}[ht!]
\begin{subfigure}{0.5 \linewidth}
\centering
\includegraphics[scale=2.25]{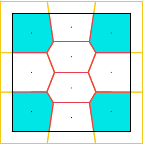}
\caption{$G(\A)$}
\end{subfigure}
\begin{subfigure}{0.5 \linewidth}
\centering
\includegraphics[scale=2.25]{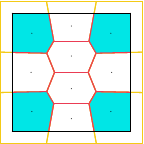}
\caption{$f_2(\A)$ with $\omega=1$} 
\end{subfigure}
\begin{subfigure}{0.5 \linewidth}
\centering
\includegraphics[scale=2.25]{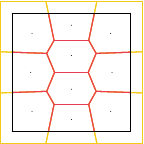}
\caption{$f_2(\A)$ with $\omega=0.1$} 
\end{subfigure}
\begin{subfigure}{0.5 \linewidth}
\centering
\includegraphics[scale=2.25]{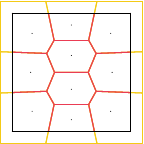}
\caption{$f_2(\A)$ with $\omega=0.01$} 
\end{subfigure}
\caption{Centroidal Voronoi tessellation with $\kappa_0=10$. In (a) we show the result of minimizing $G(\A)$. In (b-d) we show the result of minimizing $f_2(\A) = \omega \, G(\A)+J^2(\A)$ with $c_2=0.5$ and $\omega\in \{1, 0.1, 0.01\}$, respectively.}
\label{fig5}
\end{figure}

\begin{figure}[ht!]
\begin{subfigure}{0.33 \linewidth}
\centering
\includegraphics[scale=1.0]{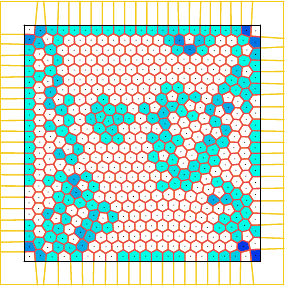}
\caption{$G(\A)$}
\end{subfigure}
\begin{subfigure}{0.33 \linewidth}
\centering
\includegraphics[scale=1.0]{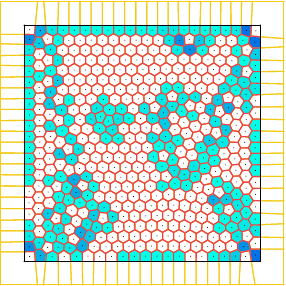}
\caption{$f_2(\A)$ with $c_2=0.2$} 
\end{subfigure}
\begin{subfigure}{0.33 \linewidth}
\centering
\includegraphics[scale=1.0]{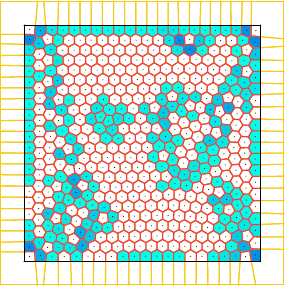}
\caption{$f_2(\A)$ with $c_2=0.3$}
\end{subfigure}
\begin{subfigure}{0.33 \linewidth}
\centering
\includegraphics[scale=1.0]{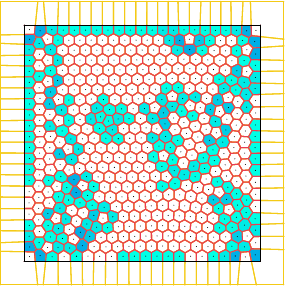}
\caption{$f_2(\A)$ with $c_2=0.4$} 
\end{subfigure}
\begin{subfigure}{0.33 \linewidth}
\centering
\includegraphics[scale=1.0]{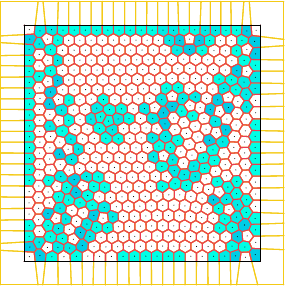}
\caption{$f_2(\A)$ with $c_2=0.5$} 
\end{subfigure}
\begin{subfigure}{0.33 \linewidth}
\centering
\includegraphics[scale=1.0]{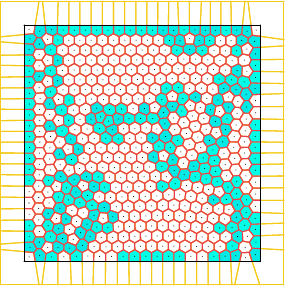}
\caption{$f_2(\A)$ with $c_2=0.6$}
\end{subfigure}
\begin{subfigure}{0.33 \linewidth}
\centering
\includegraphics[scale=1.0]{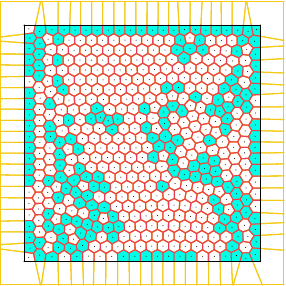}
\caption{$f_2(\A)$ with $c_2=0.7$} 
\end{subfigure}
\begin{subfigure}{0.33 \linewidth}
\centering
\includegraphics[scale=1.0]{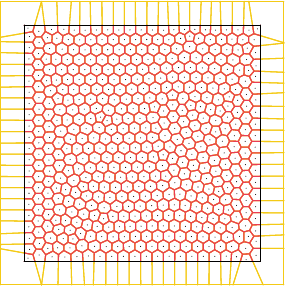}
\caption{$f_2(\A)$ with $c_2=0.8$} 
\end{subfigure}
\caption{Centroidal Voronoi tessellation with $\kappa_0=500$. In (a) we show the result of minimizing the function $G(\A)$. The darker the cell, the more unbalanced the sizes of its edges. In (b-h), preserving the meaning of the colors, we present the resulting diagrams by minimizing $f_2$ with $\omega=0.01$ and varying $c_2\in \{ 0.2, 0.3,\dots, 0.8 \}$.}
\label{fig6}
\end{figure}

\begin{figure}[ht!]
\begin{subfigure}{0.33 \linewidth}
\centering
\includegraphics[scale=1.0]{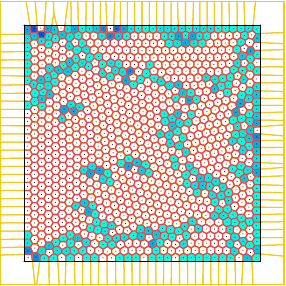}
\caption{$G(\A)$}
\end{subfigure}
\begin{subfigure}{0.33 \linewidth}
\centering
\includegraphics[scale=1.0]{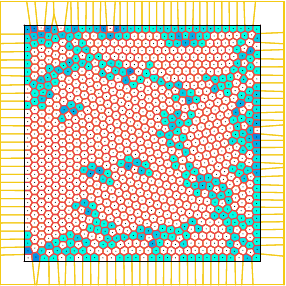}
\caption{$f_2(\A)$ with $c_2=0.3$} 
\end{subfigure}
\begin{subfigure}{0.33 \linewidth}
\centering
\includegraphics[scale=1.0]{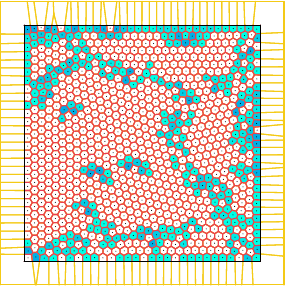}
\caption{$f_2(\A)$ with $c_2=0.4$}
\end{subfigure}
\begin{subfigure}{0.33 \linewidth}
\centering
\includegraphics[scale=1.0]{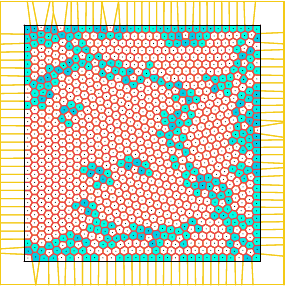}
\caption{$f_2(\A)$ with $c_2=0.5$} 
\end{subfigure}
\begin{subfigure}{0.33 \linewidth}
\centering
\includegraphics[scale=1.0]{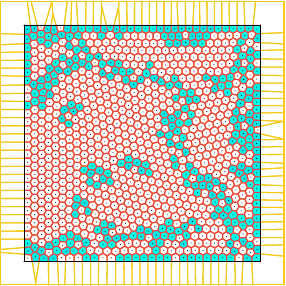}
\caption{$f_2(\A)$ with $c_2=0.6$}
\end{subfigure}
\begin{subfigure}{0.33 \linewidth}
\centering
\includegraphics[scale=1.0]{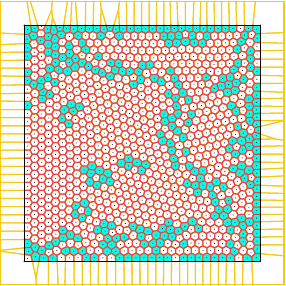}
\caption{$f_2(\A)$ with $c_2=0.7$} 
\end{subfigure}
\begin{subfigure}{0.33 \linewidth}
\centering
\includegraphics[scale=1.0]{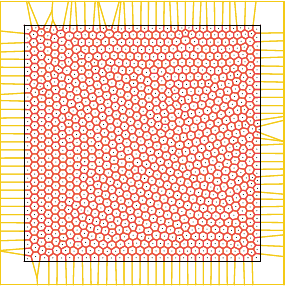}
\caption{$f_2(\A)$ with $c_2=0.8$}
\end{subfigure}
\captionsetup{list=no}
\caption{Centroidal Voronoi tessellations with $\kappa_0=1000$. In (a) we show the result of minimizing $G(\A)$. The darker the cell, the more unbalanced the sizes of its edges. In (b-g), preserving the meaning of the colors, we present the diagrams that result from minimizing  $f_2(\A)$ with $\omega=0.01$ and varying $c_2\in \{ 0.3, 0.4, \dots, 0.8 \}$.}
\label{fig7}
\end{figure}

Figure \ref{fig8} analyzes ten different diagrams with $\kappa_0=1000$ obtained by minimizing $G$ alone and the function $f_2$ with $\omega=0.01$ and $c_2 \in \{ 0.1, 0.2, \dots, 0.9\}$. For a given solution $\A^\star$, the figure shows the proportion of cells $V_i(\A^\star)$ satisfying $J^2(\A^\star)=0$ as a function of $c \in [0,1]$. Figure~\ref{fig8} shows, for example, that the statement ``all my edges are at least $10\%$ the average size of my edges'' is true for $100\%$ of the cells in any of the ten solutions and that the statement ``all my edges are at least $50\%$ the average size of my edges'' is true for slightly more than $95\%$ of the cells in the solutions computed by minimizing $G$ or minimizing $f_2$ with $c_2 \in \{0.1, 0.2, 0.3, 0.4 \}$. The zoom in the figure shows that when we minimize $f_2(\A)$ with $c_2 \leq 0.8$, the statement ``all my edges are at least $100\% \times c_2$ the average size of my edges'' is true for all the cells. When $c_2=0.9$, the geometric constraints are too restrictive and the solution does not satisfy the desired property. This corresponds to the fact that the corresponding curve falls below 1 for $c<0.9$. As a general observation, all the ten curves look very similar when $c$ varies from 0 to 0.35. This is because, in general, when a diagram is built by minimizing $G$ alone, there are only a few cells with small edges. The ``difference'' between the curves shows that these few undesired edges are eliminated when $f_2$ is minimized for increasing values of $c_2$. The curve relative to the diagram obtained by minimizing $G$ overlaps with that of minimizing $f_2$ with $c_2=0.1$. These curves are equal because, when $G$ is minimized, no cell has an edge whose size is less than 10\% of the mean.

\begin{figure}[ht!]
\begin{center}
\includegraphics[scale=0.5]{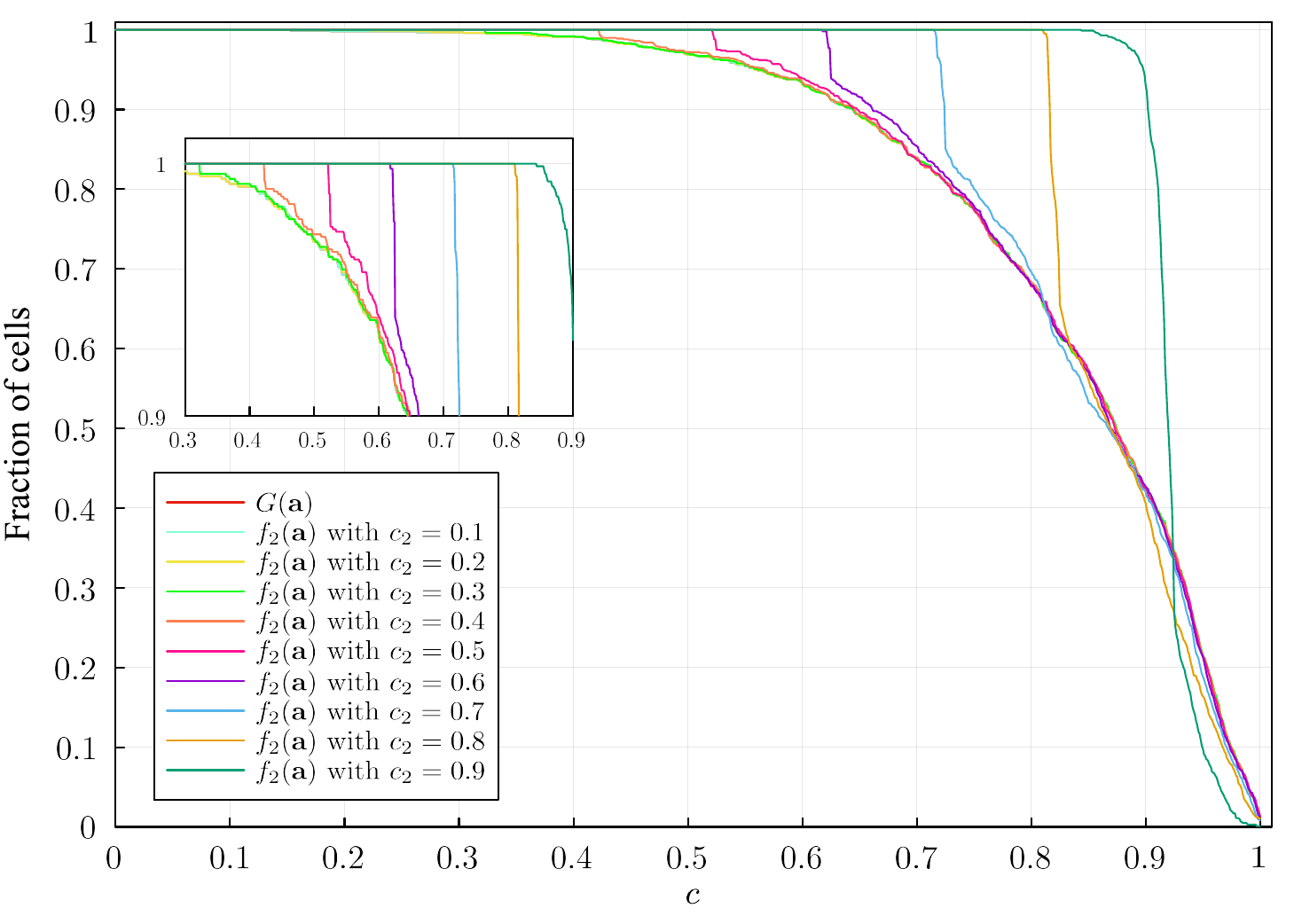}
\end{center}
\caption{This figure analyzes the solutions with $\kappa_0=1000$ found when minimizing $G$ alone and the function $f_2(\A)= \omega \, G(\A)+J^2(\A)$ with $\omega=0.01$, varying $c_2\in \{ 0.1, \dots, 0.9 \}$. For each solution, the figure shows, as a function of $c$, the proportion of cells that satisfy the statement ``all my edges are at least $100\%\times c$ the average size of my edges.''}
\label{fig8}
\end{figure}


\subsection{Density-based centroidal Voronoi tessellations} \label{CVT S.C.V}

In this section, we deal with the construction of CVTs with cells whose size is determined by a function $\psi: A \to \R$. For this, following \cite{birgin:01}, we consider the merit function given by
\[
f_3(\A)=G(\A)+ \omega \, J^3(\A),
\]
where
\[
J^3(\A) := \frac{1}{\kappa_0} \sum^{\kappa_0}_{i=1}[J^3_i(\A)]^2 \text{ with } J_i^3(\A):=\left(\int_{V_i(\A)}\,dx\right)/\left(\frac{1}{\kappa_0}\int_A \, dx \right)-\psi(a_i),
\]
where $\omega \geq 0$ is given. The function $\psi$ has the role of dictating the desired ratio between the area of the cell $V_i(\A)$ and the mean area of the cells. The gradient $J^3(\A)$ depends on $\nabla J^3_i(\A)$ and $\nabla J^3_i(\A) = \nabla J^1_i(\A) - \nabla_{a_i}\psi(a_i)\cdot \delta a_i$.   A difficulty with the function $J^3$ thus defined is that the sum of the desired areas does not necessarily equal the total area of the region $A$. As a consequence, $J^3$ is expected not to vanish in its global minimizer, which makes it impractical to establish that the global minimizer has been reached. Another option would be to consider a merit function like the one defined in Section~\ref{CVT I.V}, but with an arbitrary non-constant density function~$\rho$. This was the approach considered in~\cite{du,liu:inria-00547936}. Its disadvantage compared to our function $J^3$ is that it relies on quadrature rules to approximate the merit function and its derivatives, while our approach still allows an exact evaluation of the integrals.

In the numerical experiments, we considered $\kappa_0=1000$ and arbitrarily defined
\begin{description}
\item[(a)] $\psi(z) = \psi_1(z) := a((\bar{z_2}-(\bar{z_1}/4)^2)^2 + (\bar{z_1}/4-1)^2) + b$, where $\bar{z} = (2z-c)/5$, $c$ is the center of the region $A$, $b=1/4$, and $a=19/16^2$; 
\item[(b)] $\psi(z) = \psi_2(z) := 0.1 + \frac{2.9}{\delta^2}\left(z_2-0.6\delta \sin(\frac{2\pi z_1}{\sqrt{\kappa_0}}) - \delta \right)^2$, where $\delta=\frac{\sqrt{\kappa_0}}{2}$;
\item[(c)] $\psi(z) = \psi_3(z) := 0.01 + 20\| z-c\|^2/r^2$, where $c$ and $r$ are the center and the radius of the circle inscribing the region $A$, respectively.
\end{description}
Table~\ref{tab9} and Figure~\ref{fig9} show the details of the optimization process and the solutions found for varying values of $\omega \in \{1, 0.1, 0.01\}$. The figures in the table show that, in some cases, the method stopped because the merit function gradient norm reached the desired value. In the other cases, the method was stopped due to ``lack of progress''. That is, the method continued as long as a decrease in the objective function was observed. If, in a successive number of iterations, progress is no longer observed, the method stops; see~\cite{nocedal:01} for details. This is not an issue in practice and this stopping criterion is as valid as any other, as the tolerance $\varepsilon_{\text{opt}}=10^{-8}$ used to stop by the gradient rule is arbitrary. In general, the figures in the table show that when $\omega=0.01$, the value of $G(\A^\star)$ is of the same order as the value we found when minimizing $G$ alone, i.e., 1.62E-01; see Table~\ref{tab2}. On the other hand, for $\omega=1$ we found values of $J^3(\A^\star)$ that are between 2 and 3 orders of magnitude smaller than those found with $\omega=0.01$, with no significant deterioration in the value of $G(\A^\star)$. Graphically, when $\omega$ is ``small'' cells tend to have more uniform areas, whereas for larger values of $\omega$ we observe cells with different areas. The case of $\psi_3$ is a little different from the other two, and what we just mentioned would be better observed considering smaller values of $\omega$.

\begin{table}[ht!]
\centering
\begin{tabular}{|c|cccccccc|}
\cline{2-9}
\multicolumn{1}{c|}{} & 
$\omega$ & $f_3(\A^\star)$ & $\| \nabla f_3(\A^\star) \|_{\infty}$ & $G(\A^\star)$ & $J^3(\A^\star)$ & it & fcnt & Time \\
\hline
\hline
\multirow{3}{*}{$\psi_1$}
& 1    & 3.00640E$-$01 & 5.1E$-$09 & 2.75839E$-$01 & 2.48011E$-$02 & 768 & 792 & 2.839\\
& 0.1  & 2.44043E$-$01 & 9.2E$-$09 & 2.14370E$-$01 & 2.96727E$-$01 & 478 & 490 & 1.140\\
& 0.01 & 1.84559E$-$01 & 9.8E$-$09 & 1.66529E$-$01 & 1.80295E$+$00 & 337 & 355 & 0.835\\
\hline
\hline
\multirow{3}{*}{$\psi_2$}
& 1    & 5.45833E$-$01 & 8.5E$-$04 & 3.17750E$-$01 & 2.28084E$-$01 & 51 & 375 & 1.395\\
& 0.1  & 2.75867E$-$01 & 1.6E$-$04 & 2.29055E$-$01 & 4.25556E$-$01 & 89 & 568 & 1.317\\ 
& 0.01 & 1.88113E$-$01 & 1.8E$-$05 & 1.66878E$-$01 & 2.12344E$+$00 & 155 & 1281 & 2.845\\
\hline
\hline
\multirow{3}{*}{$\psi_3$}
& 1    & 2.07753E$+$00 & 7.2E$-$05 & 1.99442E$+$00 & 8.31023E$-$02 & 678 & 807 & 3.187\\
& 0.1  & 1.82859E$+$00 & 3.8E$-$03 & 1.59789E$+$00 & 1.20786E$+$00 & 1396 & 1521 & 3.638\\ 
& 0.01 & 8.37667E$-$01 & 9.4E$-$09 & 5.18745E$-$01 & 3.18922E$+$01 & 1990 & 2067 & 4.640\\
\hline
\end{tabular}
\caption{Details of the process of minimizing $f_3(\A) = G(\A) + \omega \, J^3(\A)$, considering the three different functions $\psi(z)$ and varying $\omega$.}
\label{tab9}
\end{table}

\begin{figure}[ht!]
\begin{subfigure}{0.33 \linewidth}
\centering
\includegraphics[scale=1.0]{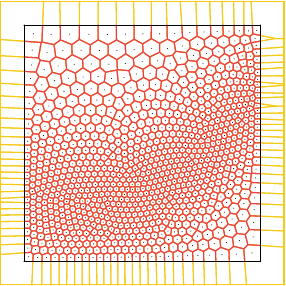}
\caption{$f_3(\A)$ with $\omega=1$}
\end{subfigure}
\begin{subfigure}{0.33 \linewidth}
\centering
\includegraphics[scale=1.0]{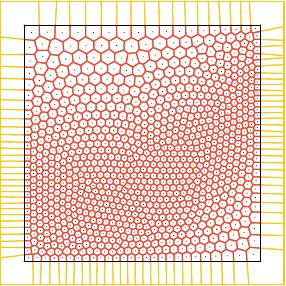}
\caption{$f_3(\A)$ with $\omega=0.1$}
\end{subfigure}
\begin{subfigure}{0.33 \linewidth}
\centering
\includegraphics[scale=1.0]{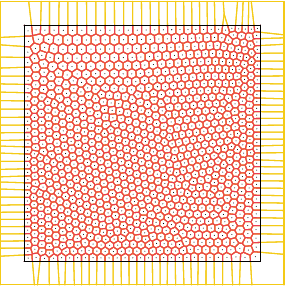}
\caption{$f_3(\A)$ with $\omega=0.01$}
\end{subfigure}
\begin{subfigure}{0.33 \linewidth}
\centering
\includegraphics[scale=1.0]{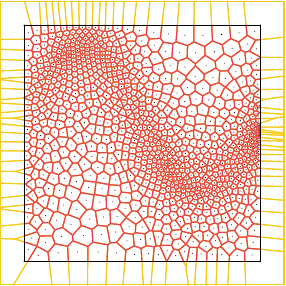}
\caption{$f_3(\A)$ with $\omega=1$}
\end{subfigure}
\begin{subfigure}{0.33 \linewidth}
\centering
\includegraphics[scale=1.0]{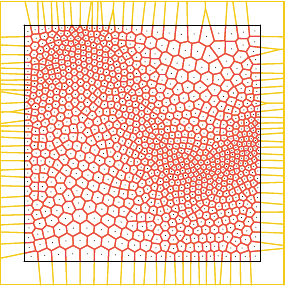}
\caption{$f_3(\A)$ with $\omega=0.1$}
\end{subfigure}
\begin{subfigure}{0.33 \linewidth}
\centering
\includegraphics[scale=1.0]{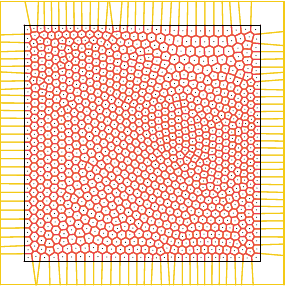}
\caption{$f_3(\A)$ with $\omega=0.01$}
\end{subfigure}
\begin{subfigure}{0.33 \linewidth}
\centering
\includegraphics[scale=1.0]{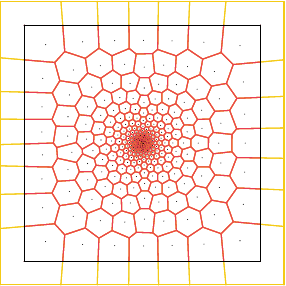}
\caption{$f_3(\A)$ with $\omega=1$}
\end{subfigure}
\begin{subfigure}{0.33 \linewidth}
\centering
\includegraphics[scale=1.0]{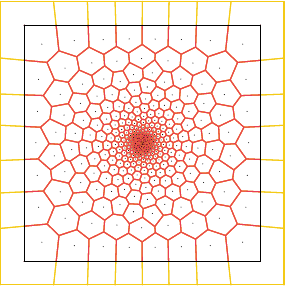}
\caption{$f_3(\A)$ with $\omega=0.1$}
\end{subfigure}
\begin{subfigure}{0.33 \linewidth}
\centering
\includegraphics[scale=1.0]{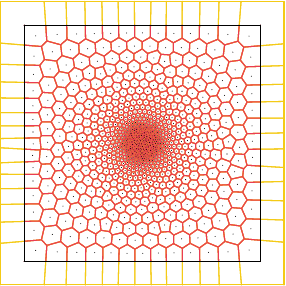}
\caption{$f_3(\A)$ with $\omega=0.01$}
\end{subfigure}
\caption{Centroidal Voronoi tessellation with $\kappa_0=1000$ constructed by seeking cells of different prescribed sizes, considering the three different functions $\psi(z)$. Pictures (a-c) correspond to $\psi_1$, (d-f) to $\psi_2$, and (g-i) to $\psi_3$.}
\label{fig9}
\end{figure}

\section{Conclusions}

In this work, we investigated the construction of centroidal Voronoi tesselations with geometric constraints. We have applied a specific case of the theory developed in \cite{birgin:01}, which provides a sensitivity analysis for Voronoi diagrams. An advantage of this approach is its unified treatment of interior and boundary edges and vertices. This analysis, which enables the computation of the derivative of any differentiable function depending on the Voronoi diagram, relies on standard nondegeneracy assumptions about the geometry.

The resulting optimization problems were easily solved with a standard optimization method, L-BFGS-B, because two arbitrary choices simplified them. The first choice was to consider a constant density function, which allowed the integral to be computed with high accuracy, without incorporating noise in the evaluation of the objective function and its derivatives. The second choice was to define a square domain $A$, which implied bound-constrained minimization problems. However, the construction of CVTs has a wide range of applications in which domains extend beyond simple box-shaped regions. The simplest case beyond rectangular domains is to consider a convex set. In this case, methods based on gradient projection would be an alternative. However, CVTs are not well defined when two or more sites coincide, and the projection operation tends to construct such points. In any case, with domains given by convex regions or more complex regions, optimization methods for general nonlinear programming (NLP) problems would be required. The problems considered in the present work, extended to arbitrary domains, represent an interesting benchmark set for existing NLP methods. In the same vein, the problems in the present work, with the inclusion of non-constant density functions, are a challenge for existing methods that deal with noise evaluations of the objective function and its gradient.

The numerical results indicate that it is possible to optimize the geometric features of CVTs while maintaining the centroidal property to a reasonable extent. In the numerical experiments, we have considered a single additional geometric constraint; if several simultaneous constraints are desired, a multi-objective optimization approach should be considered. A natural extension of this work is its application to large grids and surface grids \cite{DGJ}. Other research directions include exploring alternative mesh quality criteria to improve the convergence of finite-difference operators, for example, by minimizing the distance between the midpoint of a cell edge and the intersection points of grid segments with the cell edge, see \cite{HRK}. For three-dimensional problems, the theoretical framework established in \cite{birgin:01} must first be extended to three dimensions.\\

\noindent
\textbf{Ethics approval and consent to participate:} not applicable.\\

\noindent
\textbf{Consent for publication:} not applicable.\\

\noindent
\textbf{Funding:} This work has been partially supported by the Brazilian agencies FAPESP (grants 2013/07375-0, 2022/05803-3 and 2023/08706-1) and CNPq (grant 302073/2022-1).\\

\noindent
\textbf{Availability of data and materials:} The datasets generated and/or analyzed during the current study are available at \url{http://ime.usp.br/~egbirgin/}.\\

\noindent
\textbf{Competing interests:} The authors declare that they have no competing interests.\\

\noindent
\textbf{Authors' contributions:} All authors contributed equally to all phases of the development of this work.\\

\noindent
\textbf{Acknowledgements:} not applicable.

\bibliographystyle{plain}

\bibliography{bfl2024}

\end{document}